\documentclass[11pt]{amsart}

\usepackage{graphicx}
\usepackage{amsmath}
\usepackage{amscd}
\usepackage{amsfonts}
\usepackage{amssymb}
\usepackage{color}
\usepackage{mathtools}
\usepackage[
hypertexnames=false, colorlinks, citecolor=red,linkcolor=blue, urlcolor=red]{hyperref}

\numberwithin{equation}{section}

\newtheorem{theorem}{Theorem}[section]
\newtheorem{lemma}[theorem]{Lemma}
\newtheorem{proposition}[theorem]{Proposition}

\newtheorem{corollary}[theorem]{Corollary}
\newtheorem{problem}[theorem]{Problem}

\theoremstyle{definition}
\newtheorem{definition}[theorem]{Definition}
\newtheorem{example}[theorem]{Example}
\newtheorem{remark}[theorem]{Remark}

\newcommand{\be}{\begin{equation}}
\newcommand{\ee}{\end{equation}}
\newcommand{\bes}{\begin{equation*}}
\newcommand{\ees}{\end{equation*}}

\newcommand{\cD}{\mathcal{D}}

\newcommand{\cK}{\mathcal{K}}

\newcommand{\cM}{\mathcal{M}}

\newcommand{\cF}{\mathcal{F}}

\newcommand{\cA}{\mathcal{A}}
\newcommand{\cB}{\mathcal{B}}

\newcommand{\cS}{\mathcal{S}}
\newcommand{\cT}{\mathcal{T}}
\newcommand{\cU}{\mathcal{U}}

\newcommand{\cW}{\mathcal{W}}

\newcommand{\AND}{\text{ and }}
\newcommand{\FOR}{\text{ for }}
\newcommand{\FORAL}{\text{ for all }}
\newcommand{\qand}{\quad\text{and}\quad}

\newcommand{\qfor}{\quad\text{for}\quad}
\newcommand{\qforal}{\quad\text{for all}\quad}

\newcommand{\bB}{\mathbb{B}}
\newcommand{\bC}{\mathbb{C}}
\newcommand{\bD}{\mathbb{D}}
\newcommand{\bN}{\mathbb{N}}

\newcommand{\bR}{\mathbb{R}}
\newcommand{\bT}{\mathbb{T}}

\newcommand{\fB}{\mathfrak{B}}
\newcommand{\fC}{\mathfrak{C}}
\newcommand{\fD}{\mathfrak{D}}

\newcommand{\re}{\operatorname{Re}}
\newcommand{\im}{\operatorname{Im}}
\newcommand{\Stab}{\operatorname{Stab}}
\newcommand{\Sym}{\operatorname{Sym}}

\newcommand{\diag}{\operatorname{diag}}
\newcommand{\conv}{\operatorname{conv}}

\newcommand{\id}{\operatorname{id}}

\newcommand{\spn}{\operatorname{span}}
\newcommand{\UCP}{\operatorname{UCP}}
\newcommand{\Bsad}{\cB(H)_{sa}^d}
\newcommand{\Wmin}[1]{\cW^{\textup{min}}_{#1}}
\newcommand{\Wmax}[1]{\cW^{\textup{max}}_{#1}}

\newcommand{\ep}{\varepsilon}

\newcommand{\ol}{\overline}

\newenvironment{spmatrix}{\left(\begin{smallmatrix}}{\end{smallmatrix}\right)}

\begin{document}

\title[Dilations and matrix set inclusions]{Dilations, inclusions of matrix convex sets, and completely positive maps}

\author[K.R. Davidson]{Kenneth R. Davidson}
\address{Pure Mathematics Department, University of Waterloo,
Waterloo, ON\ N2L 3G1, Canada}
\email{krdavids@uwaterloo.ca \vspace{-2ex}}

\author[A. Dor-On]{Adam Dor-On}
\email{adoron@uwaterloo.ca}

\author[O.M. Shalit]{Orr Moshe Shalit}

\address{Faculty of Mathematics\\
Technion - Israel Institute of Technology\\
Haifa\; 3200003\\
Israel}
\email{oshalit@tx.technion.ac.il \vspace{-2ex}}

\author[B. Solel]{Baruch Solel}
\email{mabaruch@tx.technion.ac.il}

\thanks{This research was supported by The Gerald Schwartz \& Heather Reisman Foundation for Technion-UWaterloo cooperation.
The work of K.R. Davidson is partially supported by an NSERC grant.
The work of Adam Dor-On was partially supported by an Ontario Trillium Scholarship.
The work of O.M. Shalit is partially supported by ISF Grant no. 474/12, and by
EU FP7/2007-2013 Grant no. 321749.}


\begin{abstract}
A matrix convex set is a set of the form $\cS = \cup_{n\geq 1}\cS_n$ (where each $\cS_n$ is a set of $d$-tuples of $n \times n$ matrices) that is invariant under UCP maps from $M_n$ to $M_k$ and under formation of direct sums.
We study the geometry of matrix convex sets and their relationship to completely positive maps and  dilation theory.
Key ingredients in our approach are polar duality in the sense of Effros and Winkler, matrix ranges in the sense of Arveson, and concrete constructions of scaled commuting normal dilation for tuples of self-adjoint operators, in the sense of Helton, Klep, McCullough and Schweighofer.

Given two matrix convex sets $\cS = \cup_{n \geq 1} \cS_n,$ and $\cT = \cup_{n \geq 1} \cT_n$, we find geometric conditions on $\cS$ or on $\cT$, such that $\cS_1 \subseteq \cT_1$ implies that $\cS \subseteq C\cT$ for some constant $C$.

For instance, under various symmetry conditions on $\cS$, we can show that $C$ above can be chosen to equal $d$, the number of variables. We also show that $C=d$ is sharp for a specific matrix convex set $\Wmax{}(\overline{\mathbb{B}_d})$ constructed from the unit ball $\mathbb{B}_d$. This led us to find an essentially unique self-dual matrix convex set $\cD$, the self-dual matrix ball, for which corresponding inclusion and dilation results hold with constant $C=\sqrt{d}$.

For a certain class of polytopes, we obtain a considerable sharpening of such inclusion results involving polar duals.
An illustrative example is that a sufficient condition for $\cT$ to contain the free matrix cube $\fC^{(d)} = \cup_n\{(T_1, \ldots, T_d) \in M_n^d :  \|T_i\| \leq 1\}$, is that $\{x \in \bR^d : \sum |x_j| \leq 1\} \subseteq \frac{1}{d}\cT_1$, i.e., that $\frac{1}{d}\cT_1$ contains the polar dual of the cube $[-1,1]^d = \fC^{(d)}_1$.

Our results have immediate implications to spectrahedral inclusion problems studied recently by Helton, Klep, McCullough and Schweig\-hofer.
Our constants do not depend on the ranks of the pencils determining the free spectrahedra in question, but rather on the ``number of variables'' $d$.
There are also implications to the problem of existence of (unital) completely positive maps with prescribed values on a set of operators.
\end{abstract}

\maketitle

\section{Introduction}

This paper was inspired by a series of papers by Helton, Klep, McCullough and others on the advantages of using matrix convex sets
when studying linear matrix inequalities (LMI).
In particular, Helton, Klep and McCullough \cite{HKM13} showed that the matricial positivity domain of an LMI contains the information needed to determine an irreducible LMI up to unitary equivalence.
We were particularly interested in a recent paper by these authors and Schweighofer \cite{HKMS15} who dilate $d$-tuples of Hermitian matrices to commuting Hermitian matrices in order to obtain bounds on inclusions of spectrahedra inside others up to a scaling.
Our work is also related to \cite{HKM14} which discusses duality.

The two central problems that attracted our attention are the following.

\begin{problem}\label{prob:interpolation}
Given two $d$-tuples of operators  $A = (A_1, \ldots, A_d) \in \cB(H)^d$ and $B = (B_1, \ldots, B_d)\in \cB(K)^d$, determine whether there exists a unital completely positive (UCP) map $\phi : \cB(H) \to \cB(K)$ such that $\phi(A_i) = B_i$ for all $i=1, \ldots, d$.
\end{problem}

\begin{problem}\label{prob:inclusion}
Given two matrix convex sets $\cS$ and $\cT$,
determine whether $\cS \subseteq \cT$.
In particular, given that $\cS_1 \subseteq \cT_1$, determine whether $\cS \subseteq C \cT$ for some constant $C$.
\end{problem}

These problems were treated by Helton, Klep, McCullough, Schweighofer, and by others.
Our goal is to approach these problems from an operator theoretic perspective, and to sharpen, generalize and unify existing results.
While Helton et al tend to deal with $d$-tuples of real Hermitian matrices, we have chosen to work in the context of $d$-tuples of matrices or operators on complex Hilbert spaces (it seems that with a little care our methods are applicable to the setting of symmetric matrices over the reals).
Moreover we simultaneously consider the Hermitian and nonself-adjoint contexts.

Duality plays a central role in our work  as well, but takes a somewhat different character.
We find that a more natural object to associate to a $d$-tuple of matrices of bounded operators is the {\em matrix range} introduced by Arveson \cite{Arv72} in the early days of non-commutative dilation theory (see Section \ref{subsec:matrix_range}).
Moreover we show that the matrix range is the polar dual of the matricial positivity domain of the associated LMI.
We provide a description of the minimal and maximal matrix convex sets determined by a convex set at the first level (in $\bC^d$).

Matrix ranges are ideally suited to describe the possible images of a $d$-tuple under UCP maps.
This was established by Arveson in the singly generated case, and easily extends to the multivariable situation.
We use this to obtain, in Section \ref{sec:UCP}, complete descriptions of when a $d$-tuple of operators can be mapped onto another by a UCP map, or a completely contractive positive (CCP) or completely contractive (CC) map.
The basic result is that there is a UCP map as in Problem \ref{prob:interpolation} if and only if $\cW(B) \subseteq \cW(A)$, where $\cW(A)$ and $\cW(B)$ denote the matrix ranges of $A$ and $B$, respectively (see Theorem \ref{thm:ExistUCP_matRan}).
This generalizes results of many authors regarding Problem \ref{prob:interpolation} \cite{AG15a, AG15b, CJW04, HJRW12, HKM13, HLPS12, J12, LP11}.

The results we obtain in Section \ref{sec:UCP} show that the matrix range $\cW(A)$ of a tuple of operators $A$ is a complete invariant of the operator system $S_A$ generated by $A$.
It is natural to ask to what extent a $d$-tuple of operators is determined by its matrix range, and this problem was resolved by Arveson in the finite-dimensional case \cite{ArvChoquet3}.
For tuples of compact operators, this problem is taken up in Section \ref{sec:minimality}.
Under a nonsingularity assumption, we show that a $d$-tuple $A$ of compact operators can always be compressed to a minimal tuple that has the same matrix range.
We characterize minimal tuples of nonsingular compact operators in terms of their multiplicity and C*-envelope, and we show that a nonsingular minimal tuple of compact operators is determined by its matrix range up to unitary equivalence.
We also consider $d$-tuples of operators that generate C*-algebras with no compact operators.
When combining our approach with Voiculescu's Weyl-von Neumann Theorem, we show that under suitable circumstances related to multiplicity,
the matrix range determines a $d$-tuple up to approximate unitary equivalence.

The remainder of the paper deals with Problem \ref{prob:inclusion}.
A key ingredient is the construction of commuting normal dilations, following \cite{HKMS15}.
In \cite[Theorem 1.1]{HKMS15}, they establish that the set of {\em all} symmetric $n \times n$ matrices dilate up to a scale factor to a family $\cF$ of commuting Hermitian contractions on a Hilbert space $H$, in the sense that there is a constant $c$ and an isometry $V : \bR^n \rightarrow H$ so that for every symmetric contraction $S \in M_n(\bR)$, there
is some $T \in \cF$ such that $c S =  V^* T V$.
In this result, it is crucial that $n$ is fixed.
In Section \ref{sec:dilations}, we provide a counterpart of this result that is independent of the ranks of the dilated operators.
For every $d$-tuple of contractive (self-adjoint) operators $X=(X_1, \ldots, X_d)$ on a Hilbert space $H$, we construct a commuting family of contractive (self-adjoint) operators $T = (T_1, \ldots, T_d)$ on a Hilbert space $K$ and an isometry $V : H \to K$ such that $\frac{1}{2d} X_i = V^* T_i V$ (or $\frac{1}{d} X_i = V^* T_i V$ for self-adjoints) for all $i$.
In the self-adjoint context we then provide variants of this dilation result under different symmetry conditions.
In particular, if $X$ lies in some matrix convex set $\cS$, then under some symmetry conditions we can construct a commuting normal dilation $T$ such that the spectrum $\sigma(\frac{1}{d}T)$ of $\frac{1}{d}T$ is contained in $\cS_1$.
This is used to obtain other scaled inclusion results for matrix convex sets.
In particular, these results can be applied to spectrahedral inclusion problems that were studied by Helton et al.
For example, in Corollary \ref{cor:mainrelax} we show that if $A$ and $B$ are two $d$-tuples of self-adjoint operators, and if $\cD_A^{sa}$ and $\cD_B^{sa}$ denote the free spectrahedra determined by $A$ and $B$, then under some symmetry assumptions on the set $\cD_A^{sa}(1) \subseteq \bR^d$, we have
\bes\label{eq:intro_DA}
\cD^{sa}_A(1) \subseteq \cD^{sa}_B(1) \Longrightarrow \cD^{sa}_A \subseteq d \cD^{sa}_B.
\ees

Section \ref{sec:examples} provides a rich class of convex sets to which the dilation and inclusion results of Section \ref{sec:dilations} can be applied.
We show that if $K$ is a convex set in $\bR^d$ that is invariant under the projection onto an isometric tight frame, then for matrix convex sets $\cS = \cup \cS_n \subseteq \cup (M_n)_{sa}^d$ and $\cT = \cup \cT_n \subseteq \cup (M_n)_{sa}^d$ such that $\cS_1 = K$, we have the implication
\be\label{eq:inclusion}
K \subseteq \cT_1 \Longrightarrow \cS \subseteq d \cT.
\ee
We use this result to show that when $K$ is the convex hull of a vertex-reflexive isometric tight frame, invariance of $K$ under projections onto the isometric tight frame defining $K$ is automatic, so that the implication in equation \eqref{eq:inclusion} holds.

An important example is the case where $K$ is the (hyper)-cube $[-1,1]^d$. We then obtain a variant of the matricial relaxation to the matrix cube problem treated by Helton et al (see \cite[Theorem 1.6]{HKMS15}).
The result then reads
\be\label{eq:cube}
[-1,1]^d \subseteq \cT_1 \Longrightarrow \fC^{(d)} \subseteq d \cT,
\ee
where $\fC^{(d)} = \{X \in \cup_n (M_n)_{sa}^d: \|X_i\| \leq 1 \FORAL i\}$ is the {\em free matrix cube}, where an analogous result hold for all real regular polytopes.

In Section \ref{sec:ball} we study an inclusion problem analogous to \eqref{eq:cube}, but we replace the matrix convex set $\fC^{(d)}$ with a self-dual matrix convex set $\fD$ defined by
\[
\fD = \big\{ X \in \cup (M_n)_{sa}^d : \big\| \sum_i X_i \otimes \ol{X}_i \big\| \leq 1 \big\}.
\]
We find that for all matrix convex sets $\cS \subseteq \cup (M_n)_{sa}^d$,
\[
 \cS_1 \subseteq \ol{\bB}_d \Longrightarrow   \cS  \subseteq \sqrt{d} \, \fD
\]
and
\[
\ol{\bB}_d \subseteq \cS_1 \Longrightarrow \fD  \subseteq \sqrt{d} \, \cS .
\]
Moreover, the constant $\sqrt{d}$ is the optimal constant in both implications (see Theorem \ref{thm:MBP}).
In fact, in both implications one may replace $\fD$ with the matrix ball $\fB = \{X \in \cup (M_n)_{sa}^d : \sum_i X_i^2  \leq I\}$.

The last section contains some additional observations regarding both problems.
In Theorem \ref{thm:PtoUCP} we show that if $A \in \cB(H)_{sa}^d$ and  $0 \in \operatorname{int}\cW(A)$ then there is a positive constant $\rho$ such that
whenever $B\in \cB(K)^d_{sa}$ is such that the map $S_A \to S_B$ given by
\[
I \mapsto I \quad , \quad A_i \mapsto B_i \,\, , \,\, i=1, \ldots, d ,
\]
is positive, then the map given by
\[
I \mapsto I \quad , \quad A_i \mapsto \rho B_i \,\, , \,\, i=1, \ldots, d ,
\]
is completely positive.

Another interesting observation that we make is that implications like \eqref{eq:inclusion} can be sharpened when one takes into account the scalar polar dual $K'$ of $K$ (see Section \ref{subsec:gen}).
As a special example, in Section \ref{subsec:cubediamond} we show that the implication \eqref{eq:cube} can be sharpened significantly to
\[
D_d \subseteq \cT_1 \Longrightarrow \fC^{(d)} \subseteq d \cT,
\]
where
\[
D_d = \{x \in \bR^d : \sum|x_j| \leq 1\}
\]
is the scalar polar dual of $[-1,1]^d$.
See Corollary \ref{cor:relaxationD} and Remark \ref{rem:relaxationD} for a dual result related to the matrix cube problem.

\medbreak
Here is a brief overview of the organization of the paper.
Section \ref{sec:background} provides background.
Polar duality is treated in Section \ref{sec:polar},
and the maximal and minimal matrix convex sets determined by the first level are described in Section \ref{sec:maxmin}.
Section \ref{sec:UCP} contains the results on completely positive maps of $d$-tuples, connecting between the existence of a UCP map between two $d$-tuples of operators and their matrix ranges.
The extent to which a $d$-tuple is determined by its matrix range is discussed in Section \ref{sec:minimality}.
In Section \ref{sec:dilations}, we establish our version of the dilation to commuting normal operators, and provide variants for various forms of symmetry.
This is applied to inclusions of matrix convex sets.
In Section \ref{sec:examples} we consider polytopes generated by vertex-reflexive isometric tight frames, and show that they provide a large class of convex sets to which the results of Section \ref{sec:dilations} can be applied.
Section \ref{sec:ball} deals with the construction of a self-dual matrix ball based on an inequality due to Haagerup.
The final section \ref{sec:relaxations} contains further discussion of matrix inclusion problems and the application to scaling of positive maps, along with inclusion results between polytopes arising from tight frames, and their duals.

\section{Matrix convex sets, free spectrahedra and matrix ranges}\label{sec:background}

In this paper, the matrix algebras $M_n$ are understood as $M_n(\bC)$.
The algebra of bounded operators on a Hilbert space $H$ is denoted by $\cB(H)$, $\cB(H)^d$ denotes $d$-tuples of operators, and $\cB(H)_{sa}^d$ denotes $d$-tuples of self-adjoint operators.
The compact operators on $H$ are denoted $\cK(H)$.
$M_n^d$ and $(M_{n})_{sa}^d$ denote $d$-tuples of matrices or self-adjoint matrices, respectively.

Let $d \in \bN$.
A {\em free set $($in $d$ free dimensions$)$} $\cS$ is a disjoint union $\cS = \cup_n \cS_n$, where $\cS_n \subseteq M^d_n$.
Containment is defined in the obvious way: we say that $\cS \subseteq \cT$ if $\cS_n \subseteq \cT_n$ for all $n$.
A free set $\cS$ is said to be {\em open/closed/convex} if $\cS_n$ is open/closed/convex for all $n$.
It is said to be {\em bounded} if there is some $C$ such that for all $n$ and all $A \in \cS_n$, it holds that $\|A_i\|\leq C$ for all $i$.
An {\em nc set} is a free set that is closed under direct sums and under simultaneous unitary conjugation.

An nc set is said to be {\em matrix convex} if it is closed under application of UCP maps, i.e.,
if $X \in \cS_n$ and $\phi \in \UCP(M_n, M_k)$, then $\phi(X) := (\phi(X_1), \ldots, \phi(X_d)) \in \cS_k$. Since every UCP map $\phi : M_n \rightarrow M_k$ has the form $\phi(X) = \sum_{i=1}^m V_i^*XV_i$ for operators $V_i : \bC^n \rightarrow \bC^k$ satisfying $\sum_{i=1}^m V_i^*V_i = I_k$ \cite[Proposition 4.7]{PauBook}, we see that the above definition of a matrix convex set coincides with the one given by Effros and Winkler \cite[Section 3]{EW97}, so that $\cS$ is closed under matrix convex combinations.

The main examples of matrix convex sets are given by {\em free spectrahedra} and {\em matrix ranges}.

\subsection{Free spectrahedra}
A {\em monic linear pencil} is a free function of the form
\[
L(x) = L_A(x) = I - \sum A_j x_j,
\]
where $A \in \cB(H)^d$.
We let $L$ act on a $d$-tuple $X=(X_1,\dots,X_d)$ in $M_n^d$ by
\[
L(X) = I \otimes I_n - \sum_{j=1}^n A_j \otimes X_j .
\]
We write
\[
\cD_A = \cD_L = \cup_n \cD_A(n) = \cup_n \cD_L(n),
\]
where
\[
\cD_L(n) = \{X = (X_j) \in M_n^d: \re L(X)\geq 0\} .
\]
The set $\cD_A$ is said to be a {\em free spectrahedron}.
Most authors use ``free spectrahedron'' for pencils with matrix coefficients --- i.e., the case where $H$ is finite dimensional --- but we allow operator coefficients.

It is of interest also to  work in the context of self-adjoint variables and coefficients, assuming that $A \in \cB(H)^d_{sa}$ and looking for self-adjoint solutions $X \in (M_n)^d_{sa}$ to $L(X) \geq 0$.
In this case we may say {\em self-adjoint spectrahedron} to clarify that we are considering this case, and we define
\[
\cD_L^{sa} = \{X = (X_j) \in (M_{n})_{sa}^d:  L(X)\geq 0\} .
\]

\begin{proposition}\label{prop:DAmatconv}
For all $A \in \cB(H)^d$,
$\cD_A$ is a closed  matrix convex set in $\cup_n M_n^d$.
If $A \in \Bsad$, then $\cD_A^{sa}$ is a closed  matrix convex set in $\cup_n (M_n)^d_{sa}$.
\end{proposition}

\begin{proof}
We treat the nonself-adjoint case.
Clearly $\cD_A$ is a closed nc set.
Suppose that  $X \in \cD_{A}(n)$ and $\phi \in \UCP(M_n,M_k)$.
Then $I_H \otimes \phi$ is UCP, and UCP maps respect real and imaginary parts.
Letting $L$ denote the monic linear pencil associated to $A$, we have $\re L(X) \geq 0$ implies that
\[ \re L(\phi(X)) = \re (I \otimes \phi)(L(X)) = (I \otimes \phi) \re L(X) \geq 0 ,\]
so $\phi(X) \in \cD_A(k)$.
\end{proof}

\begin{example}
The {\em$(d$-dimensional$)$  matrix cube} $\fC^{(d)}$ is the self-adjoint free spectrahedron $\cD_C^{sa}$ determined by the tuple of $2d \times 2d$ matrices $C_j = \left(\begin{smallmatrix} \!E_{jj}\! & 0 \\ 0 & \!-E_{jj}\! \end{smallmatrix}\right)$, where $E_{jj}$ is the diagonal $d \times d$ matrix with $1$ at the $j$th place and $0$s elsewhere.
Then $X\in \cD^{sa}_C$ if and only if
{\allowdisplaybreaks
\begin{align*}
 0 & \le I - \sum_{j=1}^d C_j \otimes X_j 
 = \sum_{j=1}^d  \begin{spmatrix}E_{jj} \otimes I & 0\\0 & E_{jj}\otimes I \end{spmatrix} -
  \begin{spmatrix}E_{jj} \otimes X_j & 0\\0 & -E_{jj}\otimes X_j \end{spmatrix} \\&
 = \begin{spmatrix}E_{jj} \otimes (I-X_j) & 0\\0 & E_{jj}\otimes (I+X_j) \end{spmatrix} .
\end{align*}
}
Hence $0 \le I \pm X_j$, or equivalently $-I \le X_j \le I$ for $1 \le j \le d$.
\end{example}

\begin{example}
The {\em$(d$-dimensional$)$ complex matrix cube}, or {\em matrix polyball} is the matrix convex set
\[
\fC\bC^{(d)} = \{X \in \cup_n M_n^d :  \|X_i\|\leq 1 \FOR 1 \le i \le d\}.
\]
Let $A_j = E_{jj} \otimes \begin{spmatrix}0 & 2\\0 & 0 \end{spmatrix}$ for $1 \le j \le d$.
Then $X \in \cD_A$ if and only if
\begin{align*}
 0 & \le \re \Big( I - \sum_{j=1}^d E_{jj} \otimes \begin{spmatrix}0 & 2\\0 & 0 \end{spmatrix}\otimes X_j \Big) 
 = \sum_{j=1}^d  E_{jj} \otimes \re \begin{spmatrix} I & -2X_j\\0 & I \end{spmatrix}  .
\end{align*}
This holds precisely when $\|X_j\|\le 1$.
\end{example}

\begin{example}\label{ex:def_ball}
The {\em real matrix ball} and the {\em complex matrix ball} are the free sets defined by
\[
\fB^{(d)} = \{ X \in \cup_n (M_n)_{sa}^d : \sum_j X_j^2 \leq I\} ,
\]
and
\[
\fB\bC^{(d)} = \{ X \in \cup_n M_n^d : \sum_j X_j X_j^* \leq I\} .
\]
The free sets $\fB$ and $\fB\bC$ are a self-adjoint free spectrahedron and a free spectrahedron, respectively, determined by the monic pencils
\[
L_1(x) = \begin{pmatrix}
1 & x_1 & x_2 & \cdots & x_d \\
x_1 & 1 & 0 & \cdots & 0 \\
x_2 & 0 &  \ddots & & 0 \\
\vdots & \vdots & & \ddots & \vdots \\
x_d & 0 & & & 1
\end{pmatrix} .
\]
and
\[
L_2(x) = \begin{pmatrix}
1 & 2 x_1 & 2 x_2 & \cdots & 2 x_d \\
0 & 1 & 0 & \cdots & 0 \\
0 & 0 &  \ddots & & 0 \\
\vdots & \vdots & & \ddots & \vdots \\
0 & 0 & & & 1
\end{pmatrix} .
\]
The details are similar to the previous two examples.
\end{example}

\subsection{Matrix ranges}\label{subsec:matrix_range}
The {\em matrix range} \cite[Section 2.4]{Arv72} of a tuple $A$ in $\cB(H)^d$ is defined to be the set
\[
\cW(A) = \cup_n \cW_n(A),
\]
where
\[
\cW_n(A) = \{(\phi(A_1), \ldots, \phi(A_d)) : \phi \in \UCP(C^*(S_A), M_n)\}.
\]
Here and below we write $S_A$ for the operator system generated by $A$, and $C^*(S_A)$ for the unital C*-algebra generated by $A$.
Note that $\cW(A)$ is contained in $\cup_n (M_n)_{sa}^d$ if and only if $A \in \cB(H)^d_{sa}$, so we do not require a separate notation for working in the self-adjoint setting.

\begin{proposition}
For all $A \in \cB(H)^d$,
$\cW(A)$ is a closed bounded matrix convex set in $\cup_n M_n^d$.
\end{proposition}

\begin{proof}
The proof is straightforward (for closedness, recall that the space $\UCP(C^*(S_A),M_n)$ is compact in the point-norm topology; see Theorem 7.4 in \cite{PauBook}).
\end{proof}

\begin{proposition}\label{prop:rangeincluded}
Let $\cS$ be a matrix convex set and suppose that $X \in \cS$.
Then $\cW(X) \subseteq \cS$. In particular, when $X$ is a normal commuting tuple of operators, $\sigma(X) \subseteq \cS_1$.
\end{proposition}

\begin{proof}
The first assertion follows from the definitions, the second from the first with the fact that $\sigma(X) \subseteq \cW_1(X)$.
\end{proof}

\begin{theorem}\label{T:W(normal)}
Let $N$ be a commuting normal $d$-tuple with $\sigma(N) = \Lambda \subset \bC^d$. Then
\[
\cW_1(N) = {\textup{conv}}(\sigma(N)).
\]
and for $n \ge 2$,
\begin{align*}
\cW_n(N) &= \ol{ \big\{ \sum_{i=1}^m \lambda^{(i)}  K_i : \lambda^{(i)} \in \Lambda,\ m \in\bN,\ K_i \in M_n^+,\ \sum K_i = I_n \big\} } \\
                &= \big\{ \sum_{i=1}^{2n^4+n^2} \lambda^{(i)}  K_i : \lambda^{(i)} \in \Lambda,\ K_i \in M_n^+,\ \sum K_i = I_n \big\}
\end{align*}
\end{theorem}

\begin{proof}
For $\lambda = (\lambda_1,...,\lambda_d) \in \bC^d$ and $K\in M_n$, we write $\lambda K = K \lambda = (\lambda_1 K,..., \lambda_d K) \in M_n^d$.
When $\lambda^{(i)} \in \Lambda = \sigma(N)$, and $K_i \in M_n^+$, satisfy
\[
\sum_{i=1}^m (K_i^{\frac{1}{2}})^* K_i^{\frac{1}{2}} = \sum_{i=1}^m K_i = I_n ,
\]
we have that
\[
\sum_{i=1}^m \lambda^{(i)} K_i  = \sum_{i=1}^m (K_i^{\frac{1}{2}})^* \lambda^{(i)} K_i^{\frac{1}{2}} \in \cW_n(N)
\]
by matrix convexity.

For the reverse inclusion, any map in $\UCP(C^*(N),M_n)$ is a compression of a representation of $C(\Lambda)$, which
by the Weyl-von Neumann-Berg Theorem (see \cite[Corollary II.4.2]{DavBook}) can be approximated by the compression of a diagonal representation.
A second approximation reduces this to a compression of a finite dimensional representation,
which has the form $\pi(f) = \sum_{i=1}^m f(\lambda^{(i)}) P_i$, where $P_i$ is a partition of $I$ into orthogonal projections.
Thus a compression has the form $\Phi(f) = \sum_{i=1}^m f(\lambda^{(i)}) K_i$ for positive $K_i$ with $\sum K_i = I_n$.
Applying this to the identity function $\id(z) = (z_1,\dots,z_d)$ yields the desired description.
When $n=1$, this set is precisely ${\textup{conv}}(\sigma(N))$.

Elimination of the closure follows from the more refined analysis of \cite[Theorem 1.4.10]{Arv69}.
It is shown that the extreme points of $\cW_n(N)$ have the form $\sum_{i=1}^m \lambda^{(i)}  K_i$, where $m \le n^2$ because of the condition that the subspaces $K_i M_n K_i$ are linearly independent.
Since $M_n$ is $2n^2$-dimensional as a \textit{real} vector space, Carath\'eodory's Theorem shows that every point in the convex hull is obtained as a combination of at most $2n^2+1$ extreme points.
Thus combining the two estimates shows that at most $p=2n^4+n^2$ terms are required. It is then a standard argument that the set of such convex combinations is already closed.
\end{proof}

\begin{corollary} \label{C:W(normal)}
If $N$ is a commuting normal $d$-tuple, then $\cW(N)$ is the smallest matrix convex set containing $\sigma(N)$.
\end{corollary}

\begin{proof}
The description of $\cW(N)$ shows that it is evidently contained in the matrix convex hull of $\sigma(N)$.
The converse is immediate from the equality $\cW_1(N) = {\textup{conv}}(\sigma(N))$.
\end{proof}

\section{Polar duality}\label{sec:polar}

Given an nc set $\cS \subseteq \cup_n M_n^d$, its {\em polar dual} \cite{EW97} is defined to be $\cS^\circ = \cup_n \cS_n^\circ$, where
\[
\cS_n^\circ = \big\{X \in M_n^d : \re \big( \sum_j A_j \otimes X_j \big) \leq I \textrm{ for all } A \in \cS \big\}.
\]
Likewise, if $\cS$ is an nc subset of $\cup_n (M_n)^d_{sa}$ then we define $\cS^\bullet = \cup_n \cS_n^\bullet$ where
\[
\cS_n^\bullet = \{X \in (M_n)^d_{sa} :  \sum_j A_j \otimes X_j \leq I \textrm{ for all } A \in \cS\}.
\]

\begin{proposition}\label{prop:WDA}
Let $A \in \cB(H)^d$.
Then
\[
(\cW(A)\cup \{0\})^\circ = \cW(A)^\circ = \cD_A.
\]
Likewise, if $A \in \cB(H)^d_{sa}$, then
\[
(\cW(A)\cup \{0\})^\bullet = \cW(A)^\bullet = \cD^{sa}_A.
\]
\end{proposition}

\begin{proof}
We will prove the second claim. The first  follows similarly by taking real parts in the appropriate places.

First note that by definition of the dual, we have $(\cW(A)\cup \{0\})^\bullet = \cW(A)^\bullet$. Next, if $X \in \cD_A^{sa}$,\vspace{.3ex} and $\phi$ is in $\UCP(C^*(S_A), M_n)$, then applying $\phi \otimes \id$ to the inequality $\sum_j A_j \otimes X_j \leq I$, we get that $\sum_j \phi(A_j) \otimes X_j \leq I$, so that $X \in \cW(A)^\bullet$.
Conversely, if $\sum_j \phi(A_j) \otimes X_j \leq I$ for all $\phi$ in $\UCP(C^*(S_A), M_n)$, then we find that $X \in \cD_A$ by letting $\phi$ range over all compressions of $\cB(H)$ to finite dimensional subspaces.
\end{proof}

\begin{lemma}\label{lem:bipolar}
If $\cS$ is a matrix convex set in $\cup_n M_n^d$ containing $0$, then
\[
\cS^{\circ \circ} = \cS.
\]
If $\cS$ is a matrix convex set in $\cup_n (M_n)_{sa}^d$ containing $0$, then
\[
\cS^{\bullet \bullet} = \cS.
\]
\end{lemma}

\begin{proof}
The first assertion is precisely the bipolar theorem of Effros and Winkler \cite[Corollary 5.5]{EW97}.

For the second assertion, we provide an additional argument.
By definition, $\cS \subseteq \cS^{\bullet\bullet}$.
To show the reverse inclusion, it will suffice to show that when considering $\cS$ and $\cS^{\bullet\bullet}$ as subsets of $\cup_n M_n^d$, we have that $\cS^{\bullet\bullet} \subseteq \cS^{\circ\circ}$, since the first assertion guarantees $\cS = \cS^{\circ \circ}$.

If $X \in (\cS^{\bullet\bullet})_n$, then $X \in (\cS^\bullet)^\circ \cap (M_n)_{sa}^d$.
This means
\[
\re \sum_j A_j \otimes X_j  = \sum_j A_j \otimes X_j\leq I
\]
for all $A \in \cS^\bullet_n$.
Now let $B \in \cS^\circ_n$.
We have that $\re \sum_j B_j \otimes Y_j \leq I$ for all $Y$ in $\cS_n$, and since every such $Y$ is self-adjoint,  this means that
\[ \sum_j \re B_j \otimes Y_j = \re \sum_j B_j \otimes Y_j ,\]
so $\re B \in \cS^\bullet_n$.
Therefore (using that $X$ is self-adjoint) we have
\[ \re \big(\sum_j B_j \otimes X_j \big)= \sum_j (\re B_j) \otimes X_j \leq I . \]
This shows that $X \in \cS^{\circ \circ}$ as required.
\end{proof}

\begin{proposition}\label{prop:DAW}
Let $A \in \cB(H)^d$.
If $0 \in \cW(A)$, then
\[
\cD_A^\circ =  \cW(A).
\]
Likewise, if $A \in \cB(H)^d_{sa}$ and $0 \in \cW(A)$, then
\[
(\cD^{sa}_A)^\bullet = \cW(A).
\]
\end{proposition}

\begin{proof}
In light of Proposition \ref{prop:WDA}, both assertions follow from the previous lemma applied to $\cS = \cW(A)$.
\end{proof}

\begin{lemma} \label{lemma:boundedness-dual-interior}
For $A\in \cB(H)^d$ the following are equivalent.
\begin{enumerate}
\item[(1)] $0 \in \operatorname{int}(\cW(A))$, in the sense that there is some $\delta >0$ such that for all $X \in (M_n)^d$, if $||X_i||<\delta$ for all $i$ then $X\in \cW(A)$.
\item[(2)] $0\in \operatorname{int}(\cW_1(A))$.
\item[(3)] $\cD_A(1)$ is bounded.
\item[(4)] $\cD_A$ is bounded, in the sense that there is some $R > 0$ such that $\|X_i\| \leq R$ for all $X \in \cD_A$ and all $i$.
\end{enumerate}
A similar result holds for $A\in \cB(H)^d_{sa}$, when we consider $\cD^{sa}_A$ instead of $\cD_A$.
\end{lemma}

\begin{proof}
We prove the lemma for self-adjoint tuples and self-adjoint domains, where the proof for the nonself-adjoint case is done similarly by considering real parts, the complex version of the Hahn-Banach separation theorem and the matrix polyball instead of the matrix cube for the proof of (2) implies (1).

(3) implies (2): If $0$ is not in $\operatorname{int}(\cW_1(A))$, it follows from convexity and the Hahn-Banach separation theorem that there are real numbers $a_1,\ldots,a_d$ (not all $0$) such that, for every $x=(x_1,\ldots,x_d)\in \cW_1(A)$, $\sum a_ix_i\geq 0$.
Thus, for every $t<0$, $\sum t a_i x_i \leq 0 < 1$ so that for every such $t$, $(ta_1,\ldots,ta_d)\in \cD^{sa}_A(1)$, contradicting (3).

(1) implies (4):
Let $\delta >0$ be such that, if $||X_i||<\delta$ for all $i$ then $X\in \cW(A)$.
Fix some $i\in \{1, \ldots, d\}$,
and let $X=(X_j)$ be defined by $X_j=0$ if $i\neq j$ and $X_i = \frac{1}{2}\delta I$.
Then $\pm X\in \cW(A)$.
Now for every $Y\in \cD_A^{sa}=\cW(A)^\bullet $, $\pm X_i\otimes Y_i =\pm \sum X_j \otimes Y_j\leq I$.
But $||X_i\otimes Y_i|| = \frac{1}{2}\delta\cdot ||Y_i||$, thus, $||Y_i||\leq 2/\delta$.
Therefore $\cD_A^{sa}$ is bounded.

(1) and (4) trivially imply (2) and (3), respectively.

(2) implies (1): Suppose $0\in \operatorname{int}(\cW_1(A))$, and let $\epsilon \cdot [-1,1]^d \subset \cW_1(A)$ be a cube of radius $\epsilon$ inside $\cW_1(A)$.
Then there is a normal $d$-tuple $N$ with $\sigma(N) = \epsilon \cdot [-1,1]^d$, so that by Corollary \ref{C:W(normal)}, $\cW(N) \subset \cW(A)$.
It then suffices to show that $0 \in \operatorname{int}(\cW(N))$ for a normal $d$-tuple with $\sigma(N) = [-1,1]^d$ the $d$-cube, but this follows from Theorem \ref{T:W(normal)}.
\end{proof}

\begin{proposition}\label{prop:LMIrange}
A closed matrix convex set $\cS \subseteq \cup_n (M_n)^d$ has the form
\[
\cS = \cW(A)
\]
for some $A \in \cB(H)^d$ if and only if $\cS$ is bounded.
$\cS$ has the form
\[
\cS = \cD_{B}
\]
for some $B \in \cB(H)^d$ if and only if $0 \in \operatorname{int}(\cS)$.
A similar result holds in the self-adjoint case.
\end{proposition}

\begin{proof}
That $\cW(A)$ and $\cD_B$ are closed matrix convex sets was noted above, and clearly $0 \in \operatorname{int}(\cD_B)$ and $\cW(A)$ is bounded.

For the converse direction in the first assertion, let $\{A^{(k)}\}_{k=1}^\infty$ be a dense sequence of points  in $\cS$ where each point appears infinitely many times, and consider a direct sum $A = \oplus_k A^{(k)}$ acting on $H = \oplus_k \bC^{n_k}$, where $n_k$ is such that $A^{(k)} \in \cS_{n_k}$.
Then $A$ is a bounded operator since $\cS$ is bounded.
Clearly, $\cS \subseteq \cW(A)$ because the latter is closed.
For the reverse inclusion, note that the intersection of $C^*(S_A)$ with the compacts on $H$ is $0$.
By the machinery of Voiculescu's theorem (e.g., \cite[Lemma II.5.2]{DavBook}), we have that if $\phi \in \UCP(C^*(S_A), M_n)$, then there is a sequence of isometries $V_i : \bC^n \to H$ such that
\[
\| \phi(A_j) - V_i^* A_j V_i \| \xrightarrow{i\to\infty} 0 \,\, , \,\, j=1, \ldots, d.
\]
But every $V_i$ has the form $V_i = (V_i^{(k)})_{k=1}^\infty$ such that $V_i^{(k)} : \bC^{n} \to \bC^{n_k}$ and $\sum_k V_i^{(n_k)*} V_i^{(n_k)} = I_n$, so $\lim_{k\to \infty} \|V_i^{(k)}\| = 0$.
Therefore
\[
V_i^* A_j V_i = \sum_k V_i^{(k)*} A_j^{(k)} V_i^{(k)} ,
\]
where the sequence converges in norm.
For a large enough finite set $F\subseteq \bN$, we have $\| \sum_{k\in F} V_i^{(k)*} V_i^{(k)} - I \| < 1$, and hence $K_F : =\sum_{k\in F} V_i^{(k)*} V_i^{(k)}$ must be invertible. Then,
\[
\sum_{k \in F} K_F^{-\frac{1}{2}*} V_i^{(k)*} V_i^{(k)} K_F^{-\frac{1}{2}} = I ,
\]
so that $\sum_{k \in F} K_F^{-\frac{1}{2}*} V_i^{(k)*} A_j^{(k)} V_i^{(k)} K_F^{-\frac{1}{2}}$ is a genuine matrix convex combination of points in $\cS$, converging (as $F$ grows) to $V_i^* A_j V_i$.
It then follows that $V_i^* A V_i \in \cS$ and so $\phi(A) \in \cS$.

For the converse direction in the second assertion, we first claim that $0 \in \operatorname{int}(\cS)$ implies that $\cS^\circ$ is bounded.
Indeed, let $\delta >0$ be such that
$||X_i||<\delta$ for all $i$ implies $X\in \cS$.
Let $X=(X_j) \in \cS$ be defined by $X_j=0$ if $j\neq k$ and $X_k = \frac{1}{2}\delta I$.
If $Y \in \cS^\circ$, then $\|\re Y_k\| \frac{\delta}{2}= \|\re \sum_j Y_j \otimes X_j\| \leq 1$, and also
$\|\im Y_k\| \frac{\delta}{2} = \|\re \sum Y_j \otimes i X_j\| \leq 1$.
Thus $\|Y_k\| \leq \frac{4}{\delta}$ for all $k$, as claimed.

Now since $\cS^\circ$ is bounded, we have $\cS^\circ  = \cW(B)$ for some $B \in \cB(H)^d$.
Thus $\cS = \cS^{\circ \circ} = (\cW(B))^{\circ} = \cD_{B}$.
\end{proof}

\section{Maximal and minimal matrix convex sets of a convex set}\label{sec:maxmin}

We wish to describe the smallest and largest matrix convex set $\cS \subseteq \cup_n (M_n)_{sa}^d$ with a given $\cS_1 \subseteq \bR^d$.
The discussion can be carried out also in the nonself-adjoint setting, we state and prove results in the self-adjoint setting for brevity.

A $d$-tuple $X \in M_n^d$ is a {\em compression} of $A \in \cB(H)^d$ if there is an isometry $V:\bC^n\to H$ so that $X_i = V^* A_i V$ for $1 \le i \le d$.
Conversely, $A$ is a {\em dilation} of $X$ if $X$ is a compression of $A$.
We will write $X \prec A$ when $X$ is a compression of $A$.
A tuple $N = (N_1, \ldots, N_d)$ will be said to be a {\em normal} tuple if $N_1, \ldots, N_d$ are normal commuting operators.
We denote by $\sigma(N)$ the joint spectrum of a normal tuple $N$.

Recall that if $C$ is a closed convex set in $\bR^d$, then $C$ is the intersection
of all half spaces of the form
\[H(\alpha,a) =\{ x \in \bR^d : \sum_{i=1}^d \alpha_i x_i \le a \} \]
which contain $C$.
Moreover, for $A \in \Bsad$, then $\cW_1(A) \subseteq H(\alpha,a) $ if and only if
\[ \sum_{i=1}^d \alpha_i A_i \le a I .\]

\begin{definition}\label{def:min and max}
Let $C$ be a closed convex set in $\bR^d$.
Define
\[ \Wmin{n}(C) = \{ X \in (M_n)_{sa}^d : X  \prec N \textrm{ normal } \text \AND \sigma(N) \subseteq C \} \]
and
\[ \Wmax{n}(C) = \{ X \in (M_n)_{sa}^d :  \sum_{i=1}^d \alpha_i X_i \le a I_n \text{  whenever  } C \subseteq H(\alpha,a) \}. \]
\end{definition}

\begin{remark}\label{rem:monotone}
Observe that $\Wmin{}(C)$ and $\Wmax{}(C)$ are matrix convex sets with $\Wmin{1}(C) = \Wmax{1}(C) = C$.
Note that if $C_1 \subseteq C_2$, then from the definitions we have
\[
\Wmin{}(C_1) \subseteq \Wmin{}(C_2) \qand \Wmax{}(C_1) \subseteq \Wmax{}(C_2) .
\]
\end{remark}

\begin{proposition}\label{min and max}
If $\cS \subseteq \cup_n (M_n)_{sa}^d$ is a closed matrix convex set with $\cS_1 = C$, then
\[ \Wmin{n}(C) \subseteq \cS_n \subseteq \Wmax{n}(C) \,\, , \,\, n \geq 1.\]
\end{proposition}

\begin{proof}
The first inclusion follows from Corollary~\ref{C:W(normal)}. Indeed, let $X \in \Wmin{n}(C)$ with normal dilation $N$ such that $\sigma(N) \subseteq C$. By Corollary~\ref{C:W(normal)} we have $\cW(N) \subseteq \cS$ , so that $X\in \cW(N) \subseteq \cS$.
The second inclusion follows from the remarks preceding Definition \ref{def:min and max}.
\end{proof}

\begin{corollary}\label{cor:normal}
Let $N$ be a normal tuple such that $\cW_1(N) = C$.
Then $\cW(N) = \Wmin{}(C)$.
\end{corollary}

\begin{proof}
$\Wmin{}(C) \subseteq \cW(N)$ by the previous proposition.
But since every UCP map $\phi: C^*(S_N) \to M_n$ has the form $\phi(T) = V^* \pi(T) V$ for an isometry $V$ and a (unital) representation $\pi$, the reverse inclusion follows by definition.
\end{proof}

\begin{corollary}\label{cor:normalcontained}
Let $N$ be a normal tuple of matrices and let $\cS$ be a matrix convex set.
Then $N \in \cS$ if and only if $\sigma(N) \subseteq \cS_1$.
\end{corollary}

\begin{proof}
If $N \in \cS$ then $\sigma(N) \subseteq \cW_1(N) \subseteq \cS_1$ by matrix convexity.
If $\sigma(N) \subseteq \cS_1$, then
$N \in \cW(N) = \Wmin{}(\cW_1(N)) \subseteq \Wmin{}(\cS_1) \subseteq \cS$.
\end{proof}

\begin{corollary}\label{cor:bounded}
If $\cS \subseteq \cup_n (M_n)_{sa}^d$ is a closed matrix convex set such that $\cS_1 \subseteq [-r,r]^d$,
then $\cS$ is bounded by $r$, in the sense that $\|X_i\| \leq r$ for all $X \in \cS$ and all $i$.
\end{corollary}

\begin{proof}
We have $\cS_1 \subseteq [-r,r]^d = \{x \in \bR^d : \pm x_i \leq r, \FORAL i\}$.
Therefore
\begin{align*}
 \cS_n &\subseteq \Wmax{n}([-r,r]^d) =
 \{ X \in (M_n)_{sa}^d :  \pm X_i \le rI \FORAL i \} = r \fC^{(d)} . \qedhere
\end{align*}
\end{proof}

If $C \subseteq \bR^d$, we let $C'$ denote the polar dual of $C$ in the usual sense:
\[
C' = \{x \in \bR^d : \sum_j x_j y_j \leq 1 \FOR y \in C \}.
\]
\begin{theorem}\label{T:polar of min}
Let $C$ be a closed convex set  in $\bR^d$.
Then
\[ \Wmin{}(C)^\bullet = \Wmax{}(C') .\]
If $0 \in C$, then
\[ \Wmax{}(C)^\bullet = \Wmin{}(C') .\]
\end{theorem}

\begin{proof}
Observe that if $X$ is a compression of $N$, then $\sum N_i \otimes Y_i \le I$ implies that
$\sum X_i \otimes Y_i \le I$.
By the spectral theorem, $ֿ\sum N_i \otimes Y_i \le I$  is a consequence of the inequalities $\sum \alpha_i  Y_i \le I$
for $\alpha = (\alpha_i) \in \sigma(N)$.
Therefore
\[
 \Wmin{}(C)^\bullet = \{ Y :  \sum \alpha_i  Y_i \le I \FOR \alpha \in C \} = \Wmax{}(C') .
\]
If $0 \in C$, then taking the polar of this, we see that
\[ \Wmin{}(C) = \Wmax{}(C') ^\bullet .\]
Now replace $C$ with $C'$ to obtain the second equality.
\end{proof}

\section{Existence of UCP maps}\label{sec:UCP}

\subsection{Existence of completely positive maps and matrix ranges}

For tuples $A \in \cB(H)^d$ and $B\in \cB(K)^d$, denote by $S_A$ and $S_B$ the operator systems spanned by the elements of the tuples $A$ and $B$, respectively.
We wish to study when there is a UCP map from $S_A$ to $S_B$ mapping $A_i$ to $B_i$.
We denote by $C^*(S_A)$ the unital C*-algebra generated by $A$, and study the existence of maps from $C^*(S_A)$ to $\cB(K)$.

Recall that due to Arveson's extension theorem there is a UCP from $S_A$ to $S_B$ sending $A$ to $B$ if and only if there is a UCP map from $C^*(S_A)$ (or $\cB(H)$) to $\cB(K)$ mapping $A$ to $B$.

\cite[Theorem 2.4.2]{Arv72} generalizes easily (from $d=1$ to general $d \in \bN$) to give the following (we record the proof for posterity).

\begin{theorem}\label{thm:ExistUCP_matRan}
Let $A \in \cB(H)^d$ and $B\in \cB(K)^d$ be $d$-tuples of operators.
\begin{enumerate}
\item Given $n \in \bN$, if there exists a unital $n$-positive map $\phi : S_A  \to S_B$ sending $A$ to $B$, then $\cW_n(B) \subseteq \cW_n(A)$.
\item There exists a UCP map $\phi : S_A  \to S_B$ sending $A$ to $B$ if and only if $\cW(B) \subseteq \cW(A)$.
If $B$ is a commuting tuple of normal operators, then this inclusion is equivalent to $\sigma(B) \subseteq \cW_1(A)$, which is also equivalent to $\cW_1(B) \subseteq \cW_1(A)$.
\item There exists a unital completely isometric map $\phi : S_A  \to S_B$ sending $A$ to $B$ if and only if $\cW(B) = \cW(A)$.
\end{enumerate}
\end{theorem}

\begin{proof}
If there is a unital $n$-positive map $\theta$ sending $A$ to $B$, then by composition we see that $\cW_n(B) \subseteq \cW_n(A)$ (recall that a map into $M_n$ is completely positive if and only if it is $n$-positive \cite[Theorem 6.1]{PauBook}).
Thus the existence of such a UCP map implies $\cW(B) \subseteq \cW(A)$.

Conversely, suppose that $\cW(B) \subseteq \cW(A)$.
Let $\{P_{\alpha}\}$ be an increasing net of finite dimensional projections converging SOT to the identity on $K$.
We will show that the net of maps $\phi_{\alpha} : S_A \rightarrow P_{\alpha} S_B P_{\alpha}$  defined by $A_i \mapsto P_{\alpha} B_i P_{\alpha}$ is a net of well defined UCP maps.
Since, for every $i$, $\lim_{\alpha} \phi_{\alpha}(A_i)$ is a bounded net converging in the weak operator topology to $B_i$, we will obtain a UCP map $\phi : S_A \to S_B$ as a point-WOT limit mapping $A$ to $B$.

Fix $\alpha$, so that in this case $P_{\alpha} S_B P_{\alpha}$ can be identified as an operator subsystem of $M_{n_{\alpha}}$ for some finite $n_{\alpha}$.
Clearly the map $T \mapsto P_{\alpha} T P_{\alpha}$ from $\cB(K)$ to $M_{n_{\alpha}}$ is UCP, so that
\[ (P_{\alpha}B_1 P_{\alpha},..., P_{\alpha}B_d P_{\alpha}) \in \cW_{n_{\alpha}}(B) \subseteq \cW_{n_{\alpha}}(A) . \]
Hence, there exists a UCP $\psi_{\alpha} : C^*(S_A) \rightarrow P_{\alpha} S_B P_{\alpha}$ such that $\psi_{\alpha}(A_i) = P_{\alpha} B_i P_{\alpha}$.
It follows that $\phi_{\alpha}:= \psi_{\alpha} |_{\cS_{A}}$ is a well defined UCP map, as required.

Now suppose that $B$ is a tuple of commuting normal operators.
If $\sigma(B) \subseteq \cW_1(A)$, the first part of Theorem \ref{T:W(normal)} implies that $\cW_1(B) \subseteq \cW_1(A)$.
If $\cW_1(B) \subseteq \cW_1(A)$, then by Proposition \ref{min and max} and Corollary \ref{cor:normal}, we deduce $\cW(B) \subseteq \cW(A)$, and the second part of the theorem is proven.

Finally, a unital and completely isometric isomorphism between operator systems is the same as a UCP map with UCP inverse, thus the final assertion follows from the second.
\end{proof}

We now look at an operator theoretic equivalent condition for matrix range containment.
We denote by $H^{(\infty)} = H \oplus H \oplus \cdots$ the infinite ampliation of a Hilbert space $H$, and for $A \in \cB(H)$ we put $A^{(\infty)} = A \oplus A \oplus \cdots$ in $\cB(H^{(\infty)})$.

\begin{theorem}\label{thm:ExistUCP_approxComp}
Let $A \in \cB(H)^d$ and $B\in \cB(K)^d$ be $d$-tuples of operators on separable Hilbert spaces. Then $\cW(B) \subseteq \cW(A)$ if and only if there are isometries $V_n : K \rightarrow H^{(\infty)}$ such that
\[
 \lim_{n\to\infty} \|B_i - V_n^* A_i^{(\infty)} V_n \| = 0 \qfor 1 \le i \le d.
\]
Moreover the isometries may be chosen so that $B_i - V_n^* A_i^{(\infty)} V_n$ are compact operators.
\end{theorem}

\begin{proof}
Suppose that $\cW(B) \subseteq \cW(A)$.
Then by Theorem \ref{thm:ExistUCP_matRan} we have a UCP map $\phi : S_A \rightarrow S_B$ sending $A_i$ to $B_i$.
By Arveson's extension theorem, there is a UCP extension $\widetilde{\phi} : C^*(S_A) \rightarrow \cB(K)$.
By Stinespring's dilation theorem, there is a separable Hilbert space $L$, an isometry $V: K \rightarrow L$ and a *-representation $\pi : C^*(S_A) \rightarrow B(L)$ such that $\widetilde{\phi}(T) = V^* \pi(T) V$.
By Voiculescu's Weyl-von Neumann theorem (e.g., \cite[Theorem II.5.3]{DavBook}), $\id^{(\infty)} \sim_\cK \id^{(\infty)} \oplus \pi$
(and if $C^*(S_A) \cap \cK(H) = \{0\}$, we can instead say that $\id \sim_\cK \id \oplus \pi$), where $\id$ is the identity representation of $C^*(S_A)$.
This means that there is a sequence of unitaries $U_n: H^{(\infty)} \to H^{(\infty)} \oplus L$ so that
\[
 \lim_{n\to\infty} \|( T^{(\infty)} \oplus \pi(T) ) - U_n T^{(\infty)} U_n^*\| = 0 \qforal T \in C^*(S_A),
\]
and moreover the differences in the limit expression are all compact operators.
Let $J$ be the natural injection of $L$ into $H^{(\infty)} \oplus L$. Then $V_n = U_n^* JV$ is a sequence of isometries of $K$ into $H^{(\infty)}$ such that
\begin{align*}
 \lim_{n\to\infty} B_i &- V_n^* A_i^{(\infty)} V_n \\
 &= \lim_{n\to\infty} V^*J^* \big( A_i^{(\infty)} \oplus \pi(A_i) \big)JV - V^*J^*U_n A_i^{(\infty)} U_n^*JV \\
 &= \lim_{n\to\infty} V^*J^* \big( A_i^{(\infty)} \oplus \pi(A_i)  - U_n A_i^{(\infty)} U_n^* \big) JV = 0;
\end{align*}
and the differences are all compact.
The converse is straightforward.
\end{proof}

The comment in the proof yields the following version when $C^*(S_A)$ contains no non-zero compact operators
(see \cite[Lemma II.5.2]{DavBook}).

\begin{corollary} \label{L:essential}
Let $A$ and $B\in \cB(H)^d$ be $d$-tuples of operators on a Hilbert space $H$. If $C^*(S_A) \cap \cK(H) = \{0\}$, then $\cW(B) \subseteq \cW(A)$ if and only if there is a sequence of isometries $V_n: K \rightarrow H$ such that
\[
 \lim_{n\to\infty} \| B_i - V_n^* A_i V_n \| = 0 \qfor 1 \le i \le d.
\]
Moreover the isometries may be chosen so that $B_i - V_n^* A_i V_n$ are compact operators.
\end{corollary}

Motivated by a similar analysis of single operators in \cite{Dav88}, we make the following definitions.

\begin{definition}
Define the distance of a $d$-tuple $X = (X_1,\dots,X_d)$ in $M_n^d$ (or $\cB(H)^d$) from a subset $\cW_n$ (in the same space) by
\[
 d_n(X,\cW_n) = \inf_{W\in\cW_n} \max_{1 \le i \le d} \| X_i - W_i \| .
\]
Then define a measure of containment of a matrix convex set $\cB$ in another matrix convex set $\cW$ by
\[
 \delta_\cW( \cB) = \sup \{ d_n(B,\cW_n) : B \in \cB_n,\ n \ge 1 \} .
\]
Also define a distance between two bounded matrix convex sets by
\[
 \delta(\cA, \cB) = \max \{ \delta_\cA(\cB), \delta_\cB(\cA) \} .
\]
We may define a semi-metric on $\cB(H)^d$ by
\[
 \rho(A,B) = \delta(\cW(A), \cW(B))  \qfor A, B \in \cB(H)^d.
\]
\end{definition}

It is easy to see that $\rho$ is symmetric and satisfies the triangle inequality. However distinct tuples can be at distance 0.
This semi-metric is blind to multiplicity; that is, $\rho(A, A^{(\infty)}) = 0$.
The following proposition which describes when this occurs is immediate from the definition and Theorem~\ref{thm:ExistUCP_matRan}.

{\samepage
\begin{proposition} \label{P:semimetric}
For $A, B \in \cB(H)^d$, the following are equivalent:
\begin{enumerate}
\item $\rho(A,B) = 0$.
\item $\cW(A) = \cW(B)$.
\item There is a completely isometric UCP map of $S_A$ onto $S_B$ sending $A$ to $B$.
\end{enumerate}
\end{proposition}
} 


We next proceed to prove an approximate version of Theorem \ref{thm:ExistUCP_matRan}.

\begin{lemma} \label{L:BD}
Let $A \in \cB(H)^d$. Select a countable dense subset
\[ \{ A^{(k)} = (A^{(k)}_{1},\dots,A^{(k)}_{d}) : k \ge 1 \} \quad\text{of}\quad \textstyle\bigcup_{n\ge1} \cW_n(A) \]
and define $\tilde A = \bigoplus_{k\ge1} A^{(k)}$.
Then $\rho(A, \tilde A) = 0$.
Moreover if $\tilde A'$ is defined by another dense subset  $\{ A'^{(k)} : k \ge 1 \}$, then $\tilde A' \sim_\cK \tilde A$.
\end{lemma}

\begin{proof}
That $\cW(\tilde A) = \cW(A)$ was established in the proof of Proposition \ref{prop:LMIrange}, so  $\rho(A, \tilde A) = 0$ (actually, in the proof of Proposition \ref{prop:LMIrange} we had an infinite multiplicity version of $\tilde{A}$, but $\rho$ is blind to multiplicity).
For the second statement, we may assume that $A$ is not a $d$-tuple of scalars, as that case is trivial.
Thus $\cW_n(A)$ is a convex set containing more than one point, and hence there are countably many of the $A^{(k)}$s and $A'^{(k)}$s in each $\cW_n(A)$.
As both are dense, given $\ep>0$, it is a routine combinatorial exercise to find a permutation $\pi$ of $\bN$ such that
\[
 \| A^{(k)} - A'^{(\pi(k))} \| < \ep \ \FORAL k \ge 1 \qand \lim_{k\to\infty} \| A^{(k)} - A'^{(\pi(k))} \| = 0.
\]
It follows that there is a unitary operator $U_\pi$ implementing this permutation so that
\[
  \| \tilde A - U_\pi \tilde A' U_\pi^* \| < \ep \qand \tilde A - U_\pi \tilde A' U_\pi^*\in \cK(H) .
\]
Thus $\tilde A' \sim_\cK \tilde A$.
\end{proof}

\begin{theorem} \label{T:approx UCP}
Let $A \in \cB(H)^d$ and $B \in \cB(K)^d$ such that
\[ \delta_{\cW(A)}(\cW(B)) = r .\]
Then there is a UCP map $\psi$ of $S_A$ into $\cB(K)$ such that
\[
 \| \psi(A_i) - B_i \| \le r \qfor 1 \le i \le d .
\]
\end{theorem}

\begin{proof}
Following Lemma~\ref{L:BD}, let $\tilde B = \bigoplus_{k\ge1} B^{(k)}$ be a block diagonal operator in $B(\tilde K)$ with $n\times n$ summands dense in $\cW_n(B)$ for each $n\ge 1$.
For each $k$, select $A^{(k)} \in \cW_{n_k}(A)$ so that $\|B_i^{(k)} - A_i^{(k)} \| \le r$ for all $i$.
Let $\tilde A = \bigoplus_{k\ge1} A^{(k)}$ in $B(\tilde K)$.
Then by Proposition~\ref{P:semimetric}, $\rho(A\oplus \tilde A, A) = 0$.
The map $\phi:B(H \oplus \tilde K) \to B(\tilde K)$ given by compression to $\tilde K$ is a UCP map that takes $A\oplus \tilde A$ to $\tilde A$.
Let $\psi_1$ be the completely isometric map of $S_A$ onto $S_{A\oplus\tilde A}$ and let $\psi_2$ be the completely isometric map of $S_{\tilde B}$ onto $S_B$ that take generators to generators.
Then letting $\tilde{\psi}_2$ be the extension of $\psi_2$ to $B(\tilde{K})$, $\psi = \tilde{\psi}_2 \phi \psi_1$ is the desired UCP map satisfying
\[
 \| \psi(A_i) - B_i \| \le r  \qfor 1 \le i \le d .\qedhere
\]
\end{proof}

\subsection{Reduction of the CC and CCP problems to the UCP problem}

It is natural to ask a similar question with the change that we seek a CC or a CCP map instead of a UCP map.
The following constructions allow us to reduce both problems to that of determining the existence of a UCP map.
For $A \in \cB(H)^d$ and $B \in \cB(K)^d$, we denote $\cA = C^*(S_A)$ and $\cB = \cB(K)$.

\begin{proposition}\label{prop:reduction1}
Let
\[
\hat{\cA} = \left\{ \begin{pmatrix} \lambda & a \\ a^* & \lambda \end{pmatrix} : \lambda \in \bC, a \in \cA\right\}
\]
and let $\hat{\cB}$ be defined similarly.
There exists a CC map $\phi: \cA \to \cB$ sending $A$ to $B$ if and only if there exists a UCP map $\hat{\phi} : \hat{\cA} \to \hat{\cB}$ sending $\hat{A}_i:=\left(\begin{smallmatrix}0 & A_i \\A_i^* & 0 \end{smallmatrix}\right)$ to $\hat{B}_i:=\left(\begin{smallmatrix}0 & B_i \\B_i^* & 0 \end{smallmatrix}\right)$.
\end{proposition}

\begin{proof}
This is a familiar reduction, see \cite[Lemma 8.1]{PauBook}.
\end{proof}

\begin{proposition}\label{prop:reduction2}
Let
\[
\tilde{\cA} = \left\{ \begin{pmatrix} a & 0 \\ 0 & \lambda \end{pmatrix} : a \in \cA, \lambda \in \bC \right\}
\]
and let $\tilde{\cB}$ be defined similarly.
There exists a CCP map $\phi: \cA \to \cB$ sending $A$ to $B$ if and only if there exists a UCP map $\tilde{\phi} : \tilde{\cA} \to \tilde{\cB}$ sending $\left(\begin{smallmatrix}A_i & 0 \\0 & 0 \end{smallmatrix}\right)$ to $\left(\begin{smallmatrix} B_i & 0  \\0 & 0 \end{smallmatrix}\right)$.
\end{proposition}

\begin{proof}
Let $1_A$ denote the unit of $\cA$.
Given a map $\phi : \cA \to \cB$ we define $\tilde{\phi} : \tilde{\cA} \to \tilde{\cB}$ by
\[
\tilde{\phi} \Big( \begin{pmatrix}A_i & 0 \\0 & \lambda \end{pmatrix}\Big) =
\begin{pmatrix} \phi(A_i) + \lambda(1_B - \phi(1_A))  & \ 0 \\0 & \ \lambda \end{pmatrix} .
\]
Then $\tilde{\phi}$ is unital, and is CP if and only if $\phi$ is CCP.
Moreover, any UCP map $\tilde{\cA} \to \tilde{\cB}$ that maps $\left(\begin{smallmatrix}A_i & 0 \\0 & 0 \end{smallmatrix}\right)$ to $\left(\begin{smallmatrix} B_i & 0  \\0 & 0 \end{smallmatrix}\right)$ must be of the form $\tilde{\phi}$ for a CCP map $\phi : \cA \to \cB$.
\end{proof}

Let $\cM$ be a bounded matrix convex set. By Proposition \ref{prop:LMIrange} and Proposition \ref{prop:WDA}, we know that $\cM^{\circ \circ} = (\cM \cup \{0\})^{\circ \circ}$, and by the Effros-Winkler separation theorem \cite[Theorem 5.1]{EW97}, we see that $\cM^{\circ \circ}$ is the smallest matrix convex set that contains $\cM$ and $\{0\}$.

\begin{corollary} \label{C:CCP map}
The following are equivalent:
\begin{enumerate}
\item there is a CCP map $\phi: S_A \to S_B$ taking $A_i$ to $B_i$.
\item $\cW(B) \subseteq \cW(A)^{\circ \circ}$
\item $\cD_A \subseteq \cD_B$.
\end{enumerate}
\end{corollary}

\begin{proof}
The equivalence of (1) and (2) follows from Proposition~\ref{prop:reduction2} and Theorem~\ref{thm:ExistUCP_matRan} (2).
The equivalence of (2) and (3) follows from Proposition~\ref{prop:WDA} and Proposition~\ref{prop:DAW}.
\end{proof}

\begin{remark}
If $0 \in \cW(A)$, then we have that $\cW(A)^{\circ \circ} = \cW(A)$. Hence, if there is a CCP map $\phi$ such that $\phi(A)=B$, then there is a UCP map with the same property.
This can be deduced from the results above, but it has a simple explanation: since $0 \in \cW(A)$, there is a UCP map $\psi$  of $S_A$ into $\bC$ so that $\psi(A)=0$. Thus given $\phi$ as above with $\phi(I) = K$, just define
\[ \Phi(X) = \phi(X) + (I-K)^{1/2} (\psi(X) \otimes I) (I-K)^{1/2} .\]
\end{remark}

\subsection{The result for commuting normal operators}

\begin{proposition}
Let $A$ and $B$ be $d$-tuples of commuting normal operators.
Then there exists a UCP map $C^*(S_A) \to \cB(K)$ sending $A$ to $B$ if and only if $\sigma(B) \subseteq {\textup{conv}}(\sigma(A))$.
There exists such a CCP map if and only if $\sigma(B) \subseteq {\textup{conv}}(\sigma(A) \cup \{0\})$.
If the operators are self-adjoint then there exists such a CC map if and only if $\sigma(B)\subseteq {\textup{conv}}(\sigma(A) \cup -\sigma(A))$.
\end{proposition}
\begin{proof}
This follows from Theorem \ref{thm:ExistUCP_matRan} and the reductions in the previous subsection, together with the observations
\[
\sigma\left(\left(\begin{smallmatrix}0 & X\\X & 0 \end{smallmatrix} \right)\right) = \sigma(X) \cup -\sigma(X)
\]
(for self-adjoint $X$) and
\[
\sigma\left(\begin{smallmatrix} X & 0 \\0 & 0 \end{smallmatrix} \right)  = \sigma(X) \cup \{0\}.
\]
The latter is obvious, and the former follows because $\left(\begin{smallmatrix}0 & X\\X & 0 \end{smallmatrix} \right) = X \otimes \left(\begin{smallmatrix}0 & 1\\1 & 0 \end{smallmatrix} \right)$ is unitary equivalent to $X \otimes \left(\begin{smallmatrix}1 & 0\\0 & -1 \end{smallmatrix} \right) = \left(\begin{smallmatrix}X & 0\\0 & -X \end{smallmatrix} \right)$.
\end{proof}

It is interesting to compare this result with \cite[Theorem 2.1]{LP11}.
By specializing to the finite dimensional, self-adjoint case, we recover their condition for the existence of a UCP or CCP map.
Recall that in that theorem, if $A$ and $B$ are $d$-tuples of commuting self-adjoint matrices, then the existence of a CP map sending $A$ to $B$ is shown to be equivalent to the existence of a matrix with nonnegative entries $D$ such that
\be\label{eq:D}
[b_{ij}] = [a_{ij}] D,
\ee
where we $a_{i1}, \ldots, a_{in}$ denote the elements on the diagonal of $A_i$ when represented with respect to a basis which simultaneously diagonalizes $A_1, \ldots, A_d$ (and $b_{ij}$ is defined likewise).
Moreover, there is  UCP map sending $A$ to $B$ if and only if such a $D$ exists which is column stochastic.
This result immediately follows from the proposition above; for example \eqref{eq:D} with $D$ column stochastic is clearly equivalent to $\sigma(B) \subseteq {\textup{conv}}(\sigma(A))$.

We also obtain an interesting characterization for the existence of a CC map sending $A$ to $B$: it is the same as \eqref{eq:D} only now the matrix $D = [d_{ij}]$ should be such that the matrix of absolute values $[|d_{ij}|]$ is column stochastic.

\subsection{Containment of free spectrahedra  and the existence of completely positive maps}

We now obtain conditions for the existence of a UCP map sending a tuple $A$ to a tuple $B$ in terms of the free spectrahedra $\cD_A$ and $\cD_B$.
The following is a generalization of \cite[Theorem 3.5]{HKM13}, to the operator, not-necessarily self-adjoint case.
An adaptation of the proof appearing there is possible, but we prefer a different approach. The following should be compared with Corollary \ref{C:CCP map} and Theorem \ref{thm:ExistUCP_matRan}.

\begin{theorem}\label{thm:maps_inclusion}
Let $A=(A_1, \ldots, A_d)$ and $B=(B_1, \ldots, B_d)$ be two $d$-tuples of operators.
Assume that $\cD_A$ is bounded.
Then
\begin{enumerate}
\item For a given $n \in \bN$, if there exists a unital $n$-positive map $\phi : S_A \to S_B$, mapping $A_i$ to $B_i$ for all $i$, then $\cD_A(n) \subseteq \cD_{B}(n)$.
\item There exists a UCP map $\phi : S_A \to S_B$, mapping $A_i$ to $B_i$ for all $i$, if and only if $\cD_A \subseteq \cD_{B}$.
\item There exists a unital completely isometric map $\phi : S_A \to S_B$, mapping $A_i$ to $B_i$ for all $i$, if and only if  $\cD_A = \cD_{B}$.
\end{enumerate}
A similar result holds in the self-adjoint case.
\end{theorem}

\begin{proof}
By Lemma \ref{lemma:boundedness-dual-interior} $\cD_A$ is bounded if and only if $0\in \operatorname{int}(\cW(A))$, so that $0\in \cW(A)$.
Hence, by Proposition \ref{prop:DAW} we know that $\cD_A \subseteq \cD_B$ if and only if $\cW(B) \subseteq \cW(A)$.
The rest is then a consequence of Theorem \ref{thm:ExistUCP_matRan}.
\end{proof}

\begin{remark}
Examining the above proof, we see how the condition on boundedness of $\cD_A$ is used: it implies that $0 \in \operatorname{int} \cW(A)$ (cf. \cite[Theorem 3.5]{HKM13}).
In fact, all we need for the proof to work is that $0 \in \cW(A)$, and it would be interesting to understand how to formulate this in terms of $\cD_A$.
It certainly may happen that $0 \in \cW(A)$ but not in the interior, for example if $A = 0$.
A version of Theorem \ref{thm:maps_inclusion} with no extra assumptions on $\cD_A$ appears in \cite[Section 2]{Zalar}.
\end{remark}

\begin{example}[$0 \in \cW(A)$ is a necessary condition]
Suppose that $A = (I,I,\ldots,I)$.
Then $\cW(A)$ contains only tuples of identity matrices, and not $0$.
Now $\cD^{sa}_A = \{X : \sum X_i \leq I\}$.
If we define $B$ to be the $d$-tuple $\oplus_{n=d}^\infty((1-1/n), \ldots, (1 - d/n))$, then $\cD^{sa}_B = \cD^{sa}_A$, but there is no UCP map (actually, no linear map) sending $A$ to $B$.
The same example shows that also in the nonself-adjoint case the condition $0 \in \cW(A)$ is necessary for the inclusion $\cD_A \subseteq \cD_B$ to imply the existence of a UCP map sending $A$ to $B$.
\end{example}

From Theorem \ref{thm:maps_inclusion} it follows that if $\cD_A$ is bounded and there is a completely isometric unital map sending $A$ to $B$, then $\partial \cD_A = \partial \cD_B$.
We can deduce the following converse condition for the existence of a completely isometric map.
This has been also observed in \cite{HKM13}, with the reasoning reversed.

\begin{corollary}
Suppose that $\cD_A$ is bounded and that $\partial \cD_A \subseteq \partial \cD_B$.
Then there exists a unital completely isometric map $\theta : S_A \to S_B$ sending $A$ to $B$.
\end{corollary}

\begin{proof}
From $\partial \cD_A \subseteq \partial \cD_B$ together with the fact that both $\cD_A$ and $\cD_B$ are convex sets with a joint point in the interior (see \cite[Lemma 3.10]{HKM13}), we deduce that $\cD_A = \cD_B$.
The conclusion now follows from Theorem \ref{thm:maps_inclusion}.
\end{proof}

\pagebreak[3]
\section{Minimality and rigidity}\label{sec:minimality}

\subsection{Minimality}

\begin{definition}
A tuple $A = (A_1,...,A_d) \in \cB(H)^d$ is said to be {\em minimal} if there is no nontrivial reducing subspace $H_0 \subset H$ such that $\widetilde{A} = (A_1 |_{H_0},..., A_d |_{H_0})$ satisfies $\cW(A) = \cW(\widetilde{A})$.
\end{definition}
In other words, $A$ is minimal if $S_A$ is not unitaly completely isometrically isomorphic to $S_{\widetilde{A}}$ for any direct summand $\widetilde{A}$ of $A$.

\begin{proposition} \label{prp:matRanMin_approxComp}
If $A = (A_1,...,A_d) \in \cB(H)^d$, then $A$ is minimal if and only if there are no two nontrivial orthogonal reducing subspaces $K_1, K_2 \subseteq H$ such that $\cW(A\big|_{K_1}) \subseteq \cW(A\big|_{K_2})$.
\end{proposition}

\begin{proof}
Suppose that there are two nontrivial orthogonal reducing subspaces $K_1,K_2$ in $H$.
Let
$A^{(i)} = (A_1|_{K_i},..., A_d |_{K_i})$  ($i=1,2$) and suppose that $\cW(A^{(1)}) \subseteq \cW(A^{(2)})$.
Set $H_0 = K_1^{\perp}$, and denote $\widetilde{A} = (A_1 |_{H_0},..., A_d|_{H_0})$.
It always holds that $\cW(\widetilde{A}) \subseteq \cW(A)$, we show the converse.
The assumption $\cW(A^{(1)}) \subseteq \cW(A^{(2)})$ is equivalent to the existence of a UCP map $\psi$ mapping $A^{(2)}$ to $A^{(1)}$.
Define a map $\phi : S_{\widetilde{A}} \rightarrow S_{A}$, by setting $\phi(\widetilde{A}_j) = A_j$.
Since $A = A^{(1)} \oplus \widetilde{A} = \psi(P_{K_2} \widetilde{A} P_{K_2}) \oplus \widetilde{A}$, $\phi$ is a UCP map.
Thus $\cW(A) \subseteq \cW(\widetilde{A})$.
Hence $\cW(A) = \cW(\widetilde{A})$, so $A$ is not minimal.

For the converse, suppose that $A$ is not minimal.
Then there exists a nontrivial reducing subspace $H_0 \subset H$ such that $\widetilde{A} = (A_1|_{H_0},...,A_d|_{H_0})$ satisfies $\cW(A) = \cW(\widetilde{A})$, therefore there is a unital completely isometric  map $\phi$ mapping $\widetilde{A}$ to $A$.
By compressing to the subspace $H_0^\perp$, we get that $P_{H_0}^{\perp} A P_{H_0}^{\perp}$ is orthogonal to $\widetilde{A}$ and is also the image of $\widetilde{A}$ under a UCP map.
\end{proof}

In the case $A$ is a $d$-tuple of compact operators, $C^*(A)$ is a C*-subalgebra of compact operators, and every minimal reducing subspace $H_{\lambda}$ of $H$ gives rise to an irreducible representation $\pi_{\lambda} : C^*(A) \rightarrow \cB(H)$ by restriction $\pi_{\lambda}(T) = T|_{H_{\lambda}}$.
So for each unitary equivalence class of irreducible representations $\zeta$, we can pick $\pi_{\zeta} : C^*(A) \rightarrow B(H_{\zeta})$, an irreducible sub-representation of the identity representation, which must be among $\{ \pi_{\lambda}\}$, as $C^*(A)$ is a subalgebra of compact operators.
Hence, the direct sum $\oplus \pi_{\zeta} : C^*(A) \rightarrow B(\oplus H_{\zeta})$ is a faithful representation of $C^*(A)$, and if we denote $H_0 := \oplus H_{\zeta} \subseteq H$, then $\widetilde{A} := A|_{H_0}$ certainly satisfies $\cW(A) = \cW(\widetilde{A})$.

We will consider below the operator system $S_A$ generated by a compact tuple $A \in \cK(H)^d$, and the unital C*-algebra $C^*(S_A)$ which it generates.
When $\dim H < \infty$, the irreducible representations of $C^*(S_A)$ are precisely those of $C^*(A)$ --- sub-representations of the identity representation.
When $\dim H = \infty$, we will treat a representation of $C^*(A)$ interchangeably with its (unique) unitization, which is a representation of $C^*(S_A)$. In this case, there is always a singular representation $\pi_0 : C^*(S_A) \to \bC$ determined by $\pi_0(I) = 1$ and $\pi_0 \big|_A = 0$.
The singular representation $\pi_0$ may or may not be equivalent to a sub-representation of the identity representation.

The C*-envelope of $S_A$ is equal to the image of $C^*(S_A)$ under the sum of all boundary representations \cite[Theorem 7.1]{ArvChoquet1} (See also \cite[Theorem 3.4]{DavKen15}).
Thus, when $H$ is infinite dimensional,
\[
C^*_e(S_A) \cong \sigma(C^*(S_A)) \oplus \bigoplus_{\zeta} \pi_\zeta(C^*(S_A)),
\]
where the second sum consists of all boundary representations that are sub-representations of the identity, and $\sigma$ is either $\pi_0$, in the case where $\pi_0$ is a boundary representation, or else it is the nil representation.
Some of the summands might be redundant, for example when $\pi_0$ is a sub-representation of the identity.
We will require a condition that will ensure that the second summand is already an isomorphic copy of the C*-envelope.
For this, let us write $\partial_{A}$ for the irreducible sub-representations of the identity which are also boundary representations for $S_A$ in $C^*(S_A)$.

\begin{definition}\label{def:nonsingular}
A tuple $A = (A_1,...,A_d) \in \cK(H)^d$ is said to be {\em nonsingular} if either $\dim H < \infty$, or if $\dim H = \infty$ and for every $n$ and every matrix $(s_{ij}) \in M_n(S_A)$,
\[
\|\pi_0(s_{ij})\| \leq \sup \left\{\|\pi(s_{ij})\| : \pi \in \partial_A \right\}.
\]
\end{definition}

\begin{remark}\label{rem:nonsingular}
The significance of the above definition is that for a nonsingular compact tuple $A$, the C*-envelope of $S_A$ is the image of $C^*(S_A)$ under a sum of sub-representations of the identity, hence it is obtained from $C^*(S_A)$ simply by compressing to an appropriate subspace.
\end{remark}

\begin{example}\label{ex:d1nonsingular}
Let $A$ be a nonzero compact selfadjoint operator on an infinite dimensional space.
The C*-algebra $C^*(S_A)$ is isomorphic to the continuous functions on $\sigma(A)$, $S_A$ is the space of all affine functions on $\sigma(A)$.
The boundary representations therefore correspond to evaluation on the largest and smallest points $m,M \in \sigma(A)$ in the spectrum, and thus the C*-envelope $C^*_e(S_A)$ is equal to the continuous functions on $\{m,M\}$.
It follows that $A$ fails to be nonsingular precisely when either all its eigenvalues are (strictly) positive, or all its eigenvalues are (strictly) negative.
\end{example}

\begin{proposition}[Sufficient conditions for nonsingularity]\label{prop:sufficient_nonsing}
The following are sufficient conditions for a tuple $A \in \cK(H)^d$ to be nonsingular:
\begin{enumerate}
\item\label{it:finite_dimension_nonsing} $\dim H < \infty$, or
\item $A$ contains the tuple $0 = (0,\ldots, 0)$ as a direct summand, or
\item\label{it:notisolext} $0$ is not an isolated extreme point of $\cW_1(A)$ (in particular, this holds whenever $0$ is not an extreme point).
\end{enumerate}
\end{proposition}
\begin{proof}
Only the third condition requires proof, and it suffices to treat the case where $\dim H = \infty$ and hence $0 = \pi_0(A) \in \cW_1(A)$.
By \cite[Theorem 2.4]{DavKen15}, every pure state $\rho$ of $S_A$ dilates to a boundary representation $\pi$ of $S_A$.
If $\rho$ does not annihilate $A$, then $\pi$ does not annihilate the compacts.
Hence it suffices to prove that under condition (\ref{it:notisolext}),
\be\label{eq:majbystates}
\|\pi_0(s_{ij})\| \leq \sup \left\{\|\rho(s_{ij})\| : \rho \textrm{ is a pure state of } S_A \textrm{ and } \rho(A) \neq 0 \right\}
\ee
for every $n$ and every matrix $(s_{ij}) \in M_n(S_A)$.
Observe that every nonzero extreme point of $\cW_1(A)$ is equal to $\rho(A)$ for some  pure state $\rho$ of $S_A$; moreover, $\rho$ is completely determined by the tuple $\rho(A) \in \bR^d$.
Now, if $0 = \pi_0(A)$ is not an extreme point of $\cW_1(A)$, then it is a convex combination of extreme points by Carath\'eodory's Theorem, so $\pi_0\big|_{S_A}$ is a convex combination of  pure states $\rho : S_A \to \bC$ such that $\rho(A) \neq 0$, so \eqref{eq:majbystates} holds.
If $0 = \pi_0(A)$ is an extreme point of $\cW_1(A)$ but not an isolated one, then $\pi_0\big|_{S_A}$ can be approximated by pure states of $S_A$ that do not annihilate $A$, so again we see that \eqref{eq:majbystates} holds.
\end{proof}

We do not prove that a nonsingular tuple $A \in \cK(H)^d$ must satisfy at least one of the conditions (\ref{it:finite_dimension_nonsing})-(\ref{it:notisolext}) listed in Proposition \ref{prop:sufficient_nonsing}.
In a previous version of the paper, the following three results (Proposition \ref{prop:MatRan-Charac}, Corollary \ref{cor:minimal_sbspce} and Theorem \ref{thm:minimal_with_same_range}) were stated without the assumption that the tuple $A$ is nonsingular.
We are grateful to Benjamin Passer who pointed out this mistake, and provided examples showing that all three results fail without this assumption.

\begin{proposition} \label{prop:MatRan-Charac}
Let $A =(A_1,...,A_d) \in \cK(H)^d$ be a $d$-tuple of compact operators. If conditions $(1)$ and $(2)$ hold, then $A$ is minimal.
\begin{enumerate}
\item
The identity representation of $C^*(A)$ is multiplicity free.
\item
The Shilov ideal of $S_A$ inside $C^*(S_A)$ is trivial.
\end{enumerate}
If $A$ is minimal and nonsingular, then conditions $(1)$ and $(2)$ hold.
\end{proposition}

\begin{proof}
Suppose that $A$ is not minimal, so there is a nontrivial reducing subspace $H_1$ such that $\cW(A) = \cW(\widetilde{A})$, where $\widetilde{A} = A |_{H_1}$.
We will show that if the identity representation of $C^*(A)$ is multiplicity free, then the Shilov ideal of $S_A$ in $C^*(S_A)$ is not trivial.
The multiplicity free assumption means that $H = \oplus_{\zeta \in \Lambda} H_\zeta$ is a direct sum of $H_{\zeta}$ for some subset $\Lambda$ of pairwise inequivalent irreducible representations of $C^*(A)$.
The proper reducing subspace $H_1 = \oplus_{\zeta \in \Lambda_1} H_\zeta$ is a direct sum of $H_{\zeta}$ for some proper subset $\Lambda_1 \subset \Lambda$.
The projection $P:= P_{H_1}$ lies in the centre of $C^*(A)''$.
This means that $J = P^{\perp} B(H) P^{\perp} \cap C^*(A)$ is a two sided ideal inside $C^*(A)$ which is proper since it contains $\oplus_{\zeta \in \Lambda\setminus\Lambda_1} \cK(H_\zeta)$. 
Note that since $A$ consists of compact operators, $P^\perp$ might not be in $C^*(A)$, 
but we can still write $J = C^*(A)P^\perp$ since $P^\perp$ is in the centre of $C^*(A)''$. 

For $S =(S_{ij}) \in M_n(S_A)$, and $T =(T_{ij})(I_n \otimes P^{\perp}) \in M_n(J)$, we have that
$$
\|S + T \| = \|S(I_n \otimes P) \oplus (S+T)(I_n \otimes P^{\perp}) \| \geq \| S(I_n \otimes P) \| = \|S\|.
$$
The last equality holds by Theorem \ref{thm:ExistUCP_matRan}, as the map sending $A_i$ to $A_i |_{H_1}$, which we identify with $A_i P$, is completely isometric.
Hence, the map induced $S_A \rightarrow C^*(S_A) / J$ is completely isometric, so that $J$ is contained in the Shilov ideal of $S_A$.
Therefore the Shilov ideal of $S_A$ in $C^*(S_A)$ is not trivial.

Now suppose that $A$ is minimal and nonsingular.
We first show that the identity representation of $C^*(A)$ is multiplicity free.
For otherwise we can find two orthogonal reducing subspaces $H_{\lambda_1}$ and $H_{\lambda_2}$ such that the restrictions $\pi_{\lambda_1}$ and $\pi_{\lambda_2}$ are unitarily equivalent.
However, as $A$ is minimal, by Proposition \ref{prp:matRanMin_approxComp}, we see that this is impossible.
Hence, we must have that $H= \oplus H_{\zeta}$ where $\{\pi_{\zeta}\} = \{\pi_{\lambda}\}$ are mutually inequivalent irreducible *-representations, and in fact $C^*(A) = \oplus_{\zeta} \cK(H_{\zeta})$ inside $\cB(H)$.

Next, we show that $C^*(S_A)$ is the C*-envelope of $S_A$.
Above we obtained that $C^*(A) = \oplus_{\zeta \in \Sigma} \cK(H_{\zeta})$, where $\Sigma$ is the set of all irreducible sub-representations of the identity representation.
By Remark \ref{rem:nonsingular}, nonsingularity of $A$ implies that
$C_e^*(S_A) = \mathbb{C} I_H +  \oplus_{\zeta \in \Lambda} \cK(H_{\zeta})$, for some subset $\Lambda \subseteq \Sigma$.
If $C_e^*(S_A) \neq C^*(S_A)$, then $\Lambda \neq \Sigma$, so $H_1 : = \oplus_{\zeta \in \Lambda} H_{\zeta}$ is a nontrivial reducing subspace for $A$, and the compression to $H_1$ is completely isometric on $S_A$.
Letting $\widetilde{A} = A \big|_{H_1}$, we find that $\cW(A) = \cW(\widetilde{A})$ in contradiction to minimality.
\end{proof}

\begin{corollary}\label{cor:minimal_sbspce}
Let $A =(A_1,...,A_d) \in \cK(H)^d$ be a nonsingular $d$-tuple of compact operators.
Then there is a reducing subspace $H_0$ such that $\widetilde{A} = (A_1 |_{H_0},..., A_d|_{H_0})$ is minimal, and $\cW(A) = \cW(\widetilde{A})$.
\end{corollary}

\begin{proof}
Let $H_0 := \oplus H_{\zeta \in \Lambda} \subseteq H$ be a direct sum of a maximal set of inequivalent subrepresentations of the identity representation of $C^*(A)$ (note that the trivial representation $0$ may be one of them), as in the discussion preceding Definition \ref{def:nonsingular}.
Write $\widetilde{A} = A|_{H_0}$.
By construction we know that $\cW(A) = \cW(\widetilde{A})$, and $C^*(\widetilde{A})$ as well as every compression of $C^*(\widetilde{A})$ to an invariant subspace is multiplicity free.
Note that $\widetilde{A}$ is also nonsingular, thus if the Shilov ideal of $S_{\widetilde{A}}$ in $C^*(S_{\widetilde{A}})$ is trivial, then we are done.

By the previous paragraph, it remains to show that every nonsingular and multiplicity free tuple $A \in \cK(S)^d$ can be reduced to a minimal direct summand.
The space $H$ decomposes as $H := \oplus_{\zeta \in \Sigma} H_{\zeta}$, corresponding to inequivalent sub-representations of the identity.
Since $A$ is nonsingular,
\[
C^*_e(S_A) = \bigoplus_{\zeta \in \Lambda} \pi_\zeta(C^*(S_A)),
\]
where $\Lambda \subseteq \Sigma$.
Define $H_0 = \oplus_{\zeta \in \Lambda} H_\zeta$.
Then the compression $\widetilde{A} = A\big|_{H_0}$ is again nonsingular, multiplicity free, and now $S_{\widetilde{A}}$ has trivial Shilov ideal in $C^*(S_{\widetilde{A}})$.
By Proposition \ref{prop:MatRan-Charac}, $\widetilde{A}$ is minimal.
\end{proof}

The following should be compared with \cite[Theorem 2.4.3]{Arv72}.

\begin{theorem}\label{thm:minimal_with_same_range}
Let $A\in\cK(H_1)^d$ and $B\in\cK(H_2)^d$ be two minimal nonsingular $d$-tuples of compact operators.
Then $\cW(A) = \cW(B)$ if and only if $A$ and $B$ are unitarily equivalent.
\end{theorem}

\begin{proof}
There is only one direction to prove, so assume that $\cW(A) = \cW(B)$.
Then there is a unital completely isometric isomorphism $\phi$ from $S_A$ to $S_B$.
This map extends to a $*$-isomorphism $\pi$ between the respective C*-envelopes, which by Proposition \ref{prop:MatRan-Charac} are $C^*(S_A)$ and $C^*(S_B)$.
We therefore have a $*$-isomorphism $\pi : C^*(A) \to C^*(B)$.
By the representation theory of C*-algebras of compact operators, $\pi = \oplus_i \pi_i$ is (up to unitary equivalence) the direct sum of irreducible sub-representations of the identity representation.
Every sub-representation of $\id_{C^*(A)}$ appears at most once, since $C^*(B)$ is multiplicity free. If the trivial representation $0$ appears as a subrepresentation of either $C^*(A)$ or $C^*(B)$ but not of the other, this would contradict minimality of one of them. Hence, we either have that $0$ is a subrepresentation of both $C^*(A)$ and $C^*(B)$, or neither of them. Since the kernel of $\pi$ is trivial, any subrepresentation of $\id_{C^*(A)}$ other than $0$ appears at least once in the decomposition of $\pi$. Thus, in any case, it is clear that $A$ and $B$ are unitarily equivalent.
\end{proof}

\begin{example}
In general, a non-compact $d$-tuple $A$ does not always have a minimal subspace as in Corollary \ref{cor:minimal_sbspce}, even if $0$ is an interior point of $\mathcal{W}_1(A)$.
Let $(\lambda_i )_{i\in \bN}$ be a dense subset of \emph{distinct} numbers on the circle $\bT$.
Define the diagonal unitary operator $T$ on $\ell^2(\bN)$ by $T(e_i) = \lambda_i e_i$.
Then $T$ is certainly normal, but has no \emph{minimal} reducing subspace $L \subset \ell^2(\bN)$ for which $\cW(T) = \cW(T|_{L})$.

Indeed, if $L$ is such a reducing subspace for $T$, then the projection $P_L$ onto it belongs to the von-Neumann algebra $W^*(T)$ generated by $T$, since $W^*(T) = \ell^{\infty}(\bN)$ is maximal abelian, and is hence equal to its own commutant inside $\cB(\ell^2(\bN))$.
Thus, $P_L$ commutes with $P_i$, where $P_i$ is the projection onto $\spn\{e_i \}$ for each $i\in \bN$.
Hence, for a fixed $i\in \bN$, we either have $P_L(e_i) = e_i$ or $P_L(e_i) = 0$.
Hence, we establish that $L = \spn\{ e_i | i \in \Lambda \}$ for some subset $\Lambda \subseteq \bN$.

Since $\cW(T) = \cW(T|_{L})$, we must have that $\sigma(T) = \sigma(T|_{L})$, so that $(\lambda_i)_{i \in \Lambda}$ must still be dense in $\bT$.
But this is impossible because then we certainly still have that $T|_L$ has a reducing subspace $L' \subset L$ such that $\cW(T|_L) = \cW(T|_{L'})$. So $T|_L$ cannot be minimal.

This example  also has the property that there are representations of $C^*(T)$ which are not unitarily equivalent, but are approximately unitarily equivalent, such as $M_z$ in $\cB(L^2(\bT))$.
As $M_z \sim_\cK T$, we have $\cW(M_z) = \cW(T)$. It also does not have a minimal subspace. Nor is any restriction of $T$ to a reducing subspace unitarily equivalent to any restriction of $M_z$ to any reducing subspace.
\end{example}

This example shows the limits of possibility, but also shine a light on a reasonable resolution.

\begin{theorem}
Let $A$ and $B$ be $d$-tuples of operators on a separable Hilbert space such that
\begin{enumerate}
\item $C^*(S_A) = C^*_e(S_A)$ and $C^*(S_B) = C^*_e(S_B)$, and
\item $C^*(A) \cap \cK(H) = \{0\} = C^*(B) \cap \cK(H)$.
\end{enumerate}
Then $A \sim_\cK B$ if and only if $\cW(A) = \cW(B)$.
\end{theorem}

\begin{proof}
One direction is trivial, so assume that $\cW(A) = \cW(B)$.
By Theorem~\ref{thm:ExistUCP_matRan}(3), there is a completely isometric map $\phi$ of $S_A$ onto $S_B$ such that $\phi(A) = B$.
Hence by the universal property of the C*-envelope, there is a $*$-isomorphism $\tilde\phi$ of $C^*_e(S_A)$ onto $C^*_e(S_B)$ extending $\phi$.
By (1), this yields a $*$-isomorphism $\hat\phi$ of $C^*(A)$ onto $C^*(B)$.
Finally by (2) and Voiculescu's Theorem (see \cite[Theorem II.5.8]{DavBook}), $\hat\phi$ is implemented by an approximate unitary equivalence.
Thus $A \sim_\cK B$.
\end{proof}

\subsection{Consequences for free spectrahedra}

\begin{definition}
Let $A \in \cB(H)^d$ and let $L_A$ be the associated monic linear pencil.
We say that $L_A$ is {\em minimal} if there is no nontrivial reducing subspace $H_0 \subset H$ such that
$\cD_{L_{\widetilde{A}}} = \cD_{L_A}$, where $\widetilde{A} = A |_{H_0}$.
\end{definition}

When $0 \in \cW(A)$, Propositions \ref{prop:WDA} and \ref{prop:DAW} imply that $A$ is minimal if and only if $L_A$ is minimal.
We therefore obtain characterizations for when a monic linear pencil $L_A$ with $A \in \cK(H)^d$ nonsingular and $0 \in \cW(A)$ is minimal, and we see that if $L_A$ is such a minimal linear pencil then $A$ is determined up to unitary equivalence by $\cD_L$.

These results are a slight generalization of Theorem 3.12 and Proposition 3.17 of \cite{HKM13}, that treated the case where $A$ is a tuple of $n \times n$ matrices over the reals.
In \cite{HKM13} the condition was that $\cD_{L_A}$ is bounded, which by Lemma \ref{lemma:boundedness-dual-interior} is slightly stronger than $0 \in \cW(A)$.
The case where $\cD_{L_A}$ is not assumed bounded (but $A$ is still assumed to be a tuple of $n \times n$ matrices) was treated recently in \cite{Zalar}.

\section{Dilations and matricial relaxation of inclusion problems}\label{sec:dilations}

\subsection{Obtaining a dilation on a finite dimensional space}

We show that once we have a dilation of a $d$-tuple of matrices to a commuting normal $d$-tuple, then we can choose our dilation to be on a finite dimensional space.

\begin{theorem} \label{theorem:fin-dim-dil}
Let $X = (X_1,\ldots, X_d) \in M_n^d$ for which there exists a commuting $d$-tuple $T=(T_1,\ldots, T_d)$ of normal operators on a Hilbert space $H$ and an isometry $V: \bC^n \rightarrow H$ such that $X_i = V^*T_iV$.
Then there is an integer $m \leq 2n^3(d+1)+1$, a $d$-tuple $Y = (Y_1,\ldots, Y_d)$ of  commuting normal operators on $\bC^m$ satisfying $\sigma(Y) \subseteq \sigma(T)$, and an isometry $W: \bC^n \rightarrow \bC^m$ such that $X_i = W^*Y_iW$ for all $i=1, \ldots, d$.
\end{theorem}

The proof is based on some ideas from \cite{Coh15,LS14,MS13}.

\begin{proof}
Suppose that $T = (T_1,...,T_d)$ and $V : \bC^n \rightarrow H$ are as in the statement of the theorem. Let $E_T$ be the joint spectral measure for $T$.
We may then write $T_i = \int_{\sigma(T)} z_idE_T(z)$, where $\sigma(T)$ is the joint spectrum of $T$, identified as a subset of $\bC^d$.
For all $i=1, \ldots, d$,
$$
X_i = V^* \left(\int_{\sigma(T)} z_idE_T(z) \right) V = \int_{\sigma(T)}z_i d (V^*E_TV)(z)
$$
and $V^*E_TV$ is a positive operator valued measure on $\sigma(T) \subseteq \bC^d$ with values in $M_n(\bC)$.
Now, the space $\spn\{z_1, \ldots, z_d\}$ of linear functions on $\sigma(T)$ is finite dimensional, and one therefore expects to have a cubature formula, that is, a finite sequence of points $w^{(1)},\ldots, w^{(M)} \in \sigma(T)$ and positive-definite matrices $A_1, \ldots, A_M$ in $M_n(\bC)$ such that $\sum_{j=1}^MA_j = I_n$ and
\be\label{eq:cub}
\int_{\sigma(T)}f(z) d (V^*E_TV)(z) = \sum_{j=1}^M f(w^{(j)}) A_j ,
\ee
for every $f \in \spn\{z_1, \ldots, z_d\}$.
Indeed, by \cite[Theorem 4.7]{Coh15} and the dimension estimates in the proof for it, when applied to the collection of functions $\{z \mapsto z_i\}_{i=1}^d$, we have $M = 2n^2(d+1) +1$ points $w^{(1)},\ldots, w^{(M)} \in \sigma(T)$ and positive-definite matrices $A_1, \ldots, A_M$ in $M_n(\bC)$ such that $\sum_{j=1}^MA_j = I_n$ so that \eqref{eq:cub} holds.
In particular,
$$
\int_{\sigma(T)}z_id(V^*E_TV)(z) = \sum_{j=1}^M w_i^{(j)} A_j
$$
for $i=1, \ldots, d$.

The sequence $A_1, \ldots, A_M$ can be considered as a positive operator valued measure on the set $\{w^{(1)}, \ldots, w^{(M)}\}$.
By Naimark's dilation theorem, this measure dilates to a spectral measure $E$ on the set $\{w^{(1)}, \ldots, w^{(M)}\}$ with values in $M_{m}(\bC)$ where $m \leq nM$ (the bound on the dimension $m$ on which the spectral measure $E$ acts follows from the proof of Naimark's theorem via Stinespring's theorem --- see Chapter 4 of \cite{PauBook}).
That is, there exist $M$ pairwise orthogonal projections $E_1,...,E_M$ on $\bC^m$ such that $\sum E_j = I_m$, and an isometry $W: \bC^n \rightarrow \bC^m$ such that $A_j = W^* E_j W$ for $j=1, \ldots, d$.

We now construct the dilation $Y$ by defining $Y_i = \sum_{j=1}^M w_i^{(j)}E_j$.
Thus, $Y = (Y_1,...,Y_d)$ is a commuting normal $d$-tuple and by construction $\sigma(Y) = \{w^{(1)},\ldots, w^{(M)}\} \subseteq \sigma(T)$.
Moreover,
$$
W^* Y_i W = \sum_{j=1}^Mw^{(j)}_i A_j = \int_{\sigma(T)}z_id(V^*E_TV)(z) = X_i ,
$$
for $i=1, \ldots, d$.
Thus, $Y$ is a commuting normal $1$-dilation for $X$ on a space of dimension at most $nM = 2n^3(d+1)+1$ with $\sigma(Y) \subseteq \sigma(T)$.
\end{proof}

\begin{remark}
One of the main results of \cite{HKMS15}, Theorem 1.1, is that there is a constant $\vartheta(n)$ such that every $d$-tuple of symmetric $n \times n$ contractive matrices $X_1, \ldots, X_d$, there is a $d$-tuple $T_1, \ldots, T_d$ of commuting self-adjoint contractions on a Hilbert space $H$ and an isometry $V: \bR^d \to H$ such that
\be\label{eq:HKMS}
\vartheta(n) X_i = V^* T_i V \,\, , \,\, i=1, \ldots, d.
\ee
A significant amount of effort in \cite{HKMS15} was dedicated to the determination of the optimal value of $\vartheta(n)$.
In fact \cite[Theorem 1.1]{HKMS15} is stronger, in that the dilation actually works for {\em all} $n \times n$ symmetric matrices, simultaneously.
It is therefore not surprising that the dilation Hilbert space $H$ in that theorem must be infinite dimensional.
It is natural to ask whether if one begins with a fixed $d$-tuple of real symmetric matrices, can one obtain \eqref{eq:HKMS} with the commuting tuple of contractions $T$ acting on a {\em finite} dimensional space $H$.
The method of Theorem \ref{theorem:fin-dim-dil} shows that this can be done, with the constant unchanged, and with control on the dimension of $H$.
\end{remark}

As a corollary to the above, we obtain a characterization of scaled dilation in terms of matrix convex set inclusion (cf. Proposition 2.1 and Theorem 8.4 of \cite{HKMS15}).

\begin{theorem} \label{thm:scaled-inclusion-min}
Let $\cS \subseteq \cup_n M_n^d$ be a closed matrix convex set, and $c>0$. The following are equivalent
\begin{enumerate}
\item
For all $X\in \cS$ there exists $T=(T_1,...,T_d) \in \cS$ commuting normal $d$-tuple such that $c T$ dilates $X$,
\item
$\cS \subseteq c  \Wmin{}(\cS_1)$,
\item
For any closed matrix convex set $\cT \subseteq \cup_n M_n^d$ we have
$$
\cS_1 \subseteq \cT_1 \implies \cS \subseteq c  \cT.
$$
\end{enumerate}
\end{theorem}

\begin{proof}
(2) implies (3) since whenever $\cS_1 \subseteq \cT_1$, we have that $\Wmin{}(\cS_1) \subseteq \Wmin{}(\cT_1) \subseteq \cT$, so that $\cS \subseteq c  \Wmin{}(\cS_1) \subseteq c  \cT$.
Conversely, (3) implies (2) since we can take $\cT = \Wmin{}(\cS_1)$ to obtain that $\cS \subset c \Wmin{}(\cS_1)$.

(1) implies (2) because whenever $X\in \cS$ is such that $c  T$ dilates $X$ for $T \in \cS$ a normal $d$-tuple, by Theorem \ref{T:W(normal)} we have that $\sigma(T)\subseteq \cW_1(T) \subseteq \cS_1$; so $T \in \Wmin{}(\cS_1)$ and thus $X \in c \Wmin{}(\cS_1)$ by matrix convexity.

Finally, we show that (2) implies (1). Indeed, suppose that $X\in \cS$, so that by the inclusion (2) there is a normal commuting $d$-tuple $N$ on some Hilbert space $H$ with $\sigma(N) \subset \cS_1$ so that $c  N$ dilates $X$.
By Theorem \ref{theorem:fin-dim-dil}, we can choose $H \cong \bC^m$ to be finite dimensional.
By Corollary \ref{cor:normalcontained} $N \in \cS$.
\end{proof}

\subsection{A constructive normal dilation for a tuple of contractions}

In this subsection we find a concrete dilations for $d$-tuples of contractions.

For an improvement of the following theorem to $C = 2d$ in the nonself-adjoint case, see Corollary \ref{cor:contractions_dilation}.

\begin{theorem}\label{thm:1dilation}
Fix $d \in \bN$.
Then there is a constant $C$ such that
for every $d$-tuple of contractions $A_1, \ldots, A_d$ on a Hilbert space $H$ there exists a Hilbert space $K$ of dimension at most $2^{2d-1} \cdot \dim H$, an isometry $V : H \to K$, and $d$ commuting normal operators $T_1, \ldots, T_d$ satisfying $\|T_i\|\leq C$  for $i=1,\ldots, d$ such that
\[
A_i = V^* T_i V \quad , \quad i=1, \ldots, d.
\]
One may choose the constant $C = 2\sqrt{2}d$.
If the $A_j$s are self-adjoint, then one may choose $C=d$ and $T$ can be taken to a tuple of self-adjoints acting on a Hilbert space of dimension $2^{d-1}\cdot \dim H$.
\end{theorem}

\begin{proof}
Suppose for the moment that the theorem is true for self-adjoint contractions, and let $A_1,...,A_d$ be a $d$-tuple of contractions.
Then by taking real and imaginary parts, we get that $\re(A_1), \im(A_1),..., \re(A_d), \im(A_d)$ is a $2d$-tuple of self-adjoint contractions.
Then there exist commuting self-adjoint $2d$-tuple $S_1,...,S_{2d}$ with $\|S_i \| \leq 2d$, acting on a space of dimension $2^{2d-1}$, that dilate $\re(A_1), \im(A_1),..., \re(A_d), \im(A_d)$.
Then for all $1\leq i \leq d$ we have that $T_i:= S_{2i-1} + i S_{2i}$ dilates $A_i$ and is an operator of norm at most $2\sqrt{2}d$.
For the norm estimate, by the spectral mapping theorem, every element in $\sigma(T_i)$ is of the form $\lambda_1 + i \lambda_2$ for $(\lambda_1,\lambda_2 ) \in \sigma(S_{2i-1},S_{2i})$ for all $1 \leq k \leq 2d$.
Hence, $|\lambda_1 + i \lambda_2| \leq 2\sqrt{2}d$, and since $T_i$ is normal, $\|T_i\| = \sup_{\lambda\in \sigma(T_i)} |\lambda| \leq 2\sqrt{2}d$.

Hence, to finish the proof, we assume $A_1, \ldots, A_d$ are all self-adjoint and find a dilation with constant $C = d$ acting on a space of dimension $2^{d-1}\cdot \dim H$.
Note that linear combinations of commuting normals give rise to commuting normals.

Let
\[ K = H \otimes \bC^{2^{d-1}} = H \otimes \bC^2 \otimes \cdots \otimes \bC^2 . \]
We identify $H$ with $H \otimes \bC e_1 \otimes \cdots \otimes \bC e_1$, where $\{e_1, e_2\}$ is the canonical basis for $\bC^2$.
Let $F = \left(\begin{smallmatrix}0 & 1 \\ 1 & 0 \end{smallmatrix} \right)$ be the flip operator on $\bC^2$.
Then define $W_i \in B(\bC^{2^{d-1}})$ by setting $W_1 = I = I_{2^{d-1}}$, and
\[
W_i = I_2 \otimes \cdots \otimes I_2 \otimes  F \otimes I_2 \otimes  \cdots \otimes I_2,
\]
where $F$ appears in the $i\!-\!1^{\text{st}}$ place, if $i>1$.
Now we let $U_j = I_H \otimes W_j$ and define
\[
T_1 = \sum_{j=1}^d (A_j \otimes I) U_j = \sum_{j=1}^d A_j \otimes W_j.
\]
Finally, for $i=2, \ldots, d$, we define $T_i = T_1 U_i$.
Note that every $T_i$ is a sum of $d$ self-adjoint contractions, so it is self-adjoint and $\|T_i\| \leq d$.
Moreover, since $U_i$ commutes with $U_j$ and with $A_j \otimes I$ for all $j$, it follows that $U_i$ commutes with $T_1$ for all $i$, so
\begin{align*}
T_i T_k &= T_1 U_i T_1 U_k
= T_1 U_k T_1 U_i  = T_k T_i.
\end{align*}
Now if $V$ denotes the isometry from $H$ to $K$ given by
\[ Vh = h \otimes e_1 \otimes \cdots \otimes e_1 ,\]
then
\[
V^* T_i V h  = V^* \big( \sum_j (A_j \otimes I) U_j U_i \big) Vh = A_i h,
\]
because only the $i$th summand will not flip one of the $e_1$s to an $e_2$.
\end{proof}

\begin{corollary}
For all $d$,
\[
\Wmax{}([-1,1]^d) \subseteq d \Wmin{}([-1,1]^d) .
\]
\end{corollary}

\begin{proof}
Select $A \in \Wmax{}([-1,1]^d)$.
In particular, $A_i=A_i^*$ and $\|A_i\|\le 1$.
The dilation of Theorem~\ref{thm:1dilation} yields commuting self-adjoints $T=(T_1,\dots,T_d)$ with $\|T_i\|\le d$ so that $A \prec T$, whence $A \in \cW(T)$.
Since $\cW(\frac1d T) \subseteq \Wmin{}([-1,1]^d)$, the result follows.
\end{proof}

\subsection{Constructive normal dilations given symmetry}
The main point of this section is to obtain the inclusion $\Wmax{}(K)$ in $d \Wmin{}(K)$ for as many compact convex sets $K$ in $\bR^d$ as we can, while providing concrete corresponding dilation theorems.

Our dilation methods unify the dilation theorem \ref{thm:1dilation}, the dilation constructed in \cite[Proposition 14.1]{HKMS15} and provide new examples for which such dilation results can be obtained.

Let $\cS \subseteq \cup_n (M_n)^d_{sa}$ be a matrix convex set. For a real $d \times d$ matrix $\gamma = [\gamma_{ij}]$, we define the set
\[
(\gamma \cS)_n = \{ (\gamma X^t)^t : X \in \cS_n \},
\]
where $X^t$ denotes the transpose of the row $X=(X_1,...,X_d)$ so that for a $d$-tuple $X = (X_1,...,X_d) \in \cS_n$, we have $(\gamma X^t)^t = (\sum_{j}\gamma_{ij}X_j)$. Clearly $\gamma \cS = \cup_n (\gamma \cS)_n \subset \cup_n (M_n)^d_{sa}$ is also a matrix convex set.

\begin{definition}
Let $\lambda :=\{\lambda^{(m)}: 1\leq m \leq k \}$ be a $k$-tuple of rank one real $d\times d$ matrices such that $I_d \in \conv\{\lambda^{(1)},\ldots,\lambda^{(k)}\}$.
We say that a matrix convex set $\cS\subseteq \cup_n (M_n)^d_{sa}$ is $\frac{1}{C}\lambda$-symmetric if $C$ is a constant so that
\[ \lambda^{(m)}\cS \subseteq C \cS \qforal 1\leq m \leq k .\]
\end{definition}

In order to prove Theorem \ref{relaxation and symmetry} we will need the following dilation result which is a generalization of Theorem \ref{thm:1dilation} (See Remark \ref{rem:getprev}).

\begin{theorem} \label{general_dilation}
Let $\cS\subseteq \cup_n (M_n)^d_{sa}$ and $\cT\subseteq \cup_n (M_n)^d_{sa}$ be matrix convex sets.
Assume that there is  a $k$-tuple of real $d\times d$ rank one matrices $\lambda$
such that $I_d \in \conv \{\lambda^{(1)},\ldots,\lambda^{(k)}\}$ and such that $\lambda^{(m)}\cS\subseteq \cT$ for all $1\leq m  \leq k$.
Then for every $X \in \cS$ there is a $d$-tuple $T=(T_1,\ldots,T_d)$ of self-adjoint matrices such that
\begin{enumerate}
\item[(1)] $\{T_1,\ldots,T_d\}$ is a commuting family of operators,
\item[(2)] $T \in \cT$,
\item[(3)] $X \prec T$ $($i.e., $T$ is a dilation for $X)$.
\end{enumerate}
\end{theorem}

\begin{proof}
Consider $X$ as a tuple of operators on a Hilbert space $H$.
Write $K=H\otimes \mathbb{C}^k$ and define $d^2$ diagonal, self-adjoint matrices $S_{i,j}$, $1\leq i,j \leq d$, by
\be\label{S}
S_{i,j}=\diag (\lambda^{(1)}_{i,j}, \ldots ,\lambda^{(k)}_{i,j}).
\ee
For every $1\leq i \leq d$, let
\be\label{T}
T_i=\sum_{j=1}^d X_j \otimes S_{i,j} \in \cB(K) .\ee

We shall now verify (1)-(3).
For (1), we fix $i,n$ and compute
\[ T_iT_n-T_nT_i=\sum_{j,m} X_jX_m \otimes (S_{i,j}S_{n,m}-S_{n,j}S_{i,m}) .\]
But the $(p,p)$ coordinate of the (diagonal) matrix
$S_{i,j}S_{n,m}-S_{n,j}S_{i,m}$ is $\lambda^{(p)}_{i,j}\lambda^{(p)}_{n,m}-\lambda^{(p)}_{n,j}\lambda^{(p)}_{i,m}$.
Since $\lambda^{(p)}$ has rank one, the last expression is $0$ and, thus, $T_iT_n=T_nT_i$, proving (1).

To prove (3), recall that $I_d \in \conv\{\lambda^{(1)},\ldots,\lambda^{(k)}\}$.
Thus there are nonnegative real numbers $\beta_1,\ldots,\beta_k$ whose sum is $1$ and $I_d=\sum_{p=1}^k \beta_p \lambda^{(p)}$.
Set $v=\sum_{p=1}^k \sqrt{\beta_p}e_p$ where $\{e_p\}$ is the standard basis of $\mathbb{C}^k$.
Then, $||v||=1$ and for $1\leq i,j \leq d$, $\langle S_{i,j}v,v\rangle=\sum_{p=1}^k \beta_p \lambda^{(p)}_{i,j}=\delta_{i,j}$.
Define an isometry $V:H\rightarrow K=H\otimes \mathbb{C}^k$ by $Vh=h\otimes v$.
Then since $V^*(X \otimes S_{ij})V = \delta_{ij}X$, we obtain
\[
V^* T_iV  =  \sum_j V^*(X_j \otimes S_{ij})V = X_i .
\]

To prove (2), rewrite $T_i=\sum_{j=1}^d X_j \otimes S_{i,j}$
as a direct sum (over $p$) of operators of the form $Y_i^{(p)} = \sum_{j=1}^d \lambda_{i,j}^{(p)} X_j $.
By matrix convexity, it suffices to show that $Y^{(p)}:=(Y_1^{(p)},\ldots,Y_d^{(p)}) \in \cT$.
But $Y^{(p)}$ is obtained from $X$ by left multiplication by $\lambda^{(p)}$.
The assumption $\lambda^{(m)}\cS\subseteq \cT$ (and the fact that $X\in \cS$), implies that $T\in \cT$ and (2) follows.
\end{proof}

We record the following corollaries of the above proof.

\begin{corollary} \label{cor:lambda-over-C-dilation}
Let $\cS\subseteq \cup_n (M_n)^d_{sa}$ be a matrix convex set.
Assume that $\cS$ is $\frac{1}{C}\lambda$-symmetric for some $C>0$ and $\lambda = \{ \lambda^{(m)}\ : 1 \leq m \leq k\}$ as above, then every $X \in \cS$ can be dilated to a commuting $d$-tuple $T$ such that $\frac{1}{C}T \in \cS$
\end{corollary}

\begin{proof}
If we take in Theorem \ref{general_dilation} the matrix convex set $\cT := C \cdot \cS$, then $\frac{1}{C}\lambda$-symmetry guarantees that $\lambda^{(m)}\cS\subseteq \cT$ for all $1\leq m  \leq k$.
\end{proof}

\begin{corollary}\label{cor:dilations_operators}
Let $X \in B(H)^d_{sa}$, and let $\lambda^{(1)}, \ldots, \lambda^{(k)}$ be a $k$-tuple of real $d\times d$ rank one matrices
such that $I_d \in \conv \{\lambda^{(1)},\ldots,\lambda^{(k)}\}$.
Then $X$ can be dilated to a commuting tuple of self-adjoint operators $T = (T_1, \ldots, T_d)$ such that
\[
\sigma(T) \subseteq \bigcup_{p=1}^k \lambda^{(p)}\cW_1(X) .
\]
\end{corollary}
\begin{proof}
Given $X \in B(H)^d_{sa}$, construct the dilation $T = (T_1, \ldots, T_d)$ as in the above proof.
Then $T_i$ is the direct sum of operators $Y^{(p)}_i \in B(H)_{sa}$ of the form $Y^{(p)}_i = \sum_{j=1}^d \lambda_{i,j}^{(p)} X_j$.
We will show that $\sigma(Y^{(p)}) \subseteq \lambda^{(p)} \cW_1(X)$ for all $p$, and, since $\sigma(N) \subseteq \cW_1(N)$ for every normal tuple $N$, it suffices to show that $\cW_1(Y^{(p)}) \subseteq \lambda^{(p)} \cW_1(X)$.

If $\phi$ is a state on $B(H)$, then $\phi(Y^{(p)}_i) = \sum_{j=1}^d \lambda_{i,j}^{(p)} \phi(X_j)$.
This shows that
\[
\left(\phi(Y^{(p)}_1) , \ldots, \phi(Y^{(p)}_d)\right)^t = \lambda^{(p)} \left(\phi(X_1) , \ldots, \phi(X_d)\right)^t.
\]
Therefore $\cW_1(Y^{(p)}) \subseteq \lambda^{(p)} \cW_1(X)$, as required.
\end{proof}

\begin{corollary}\label{cor:contractions_dilation}
For every $d$-tuple of contractions $A_1, \ldots, A_d$ on a Hilbert space $H$ there exists a Hilbert space $K$, an isometry $V : H \to K$, and $d$ commuting normal operators $T_1, \ldots, T_d$ satisfying $\|T_i\|\leq 2d$  for $i=1,\ldots, d$ such that
\[
A_i = V^* T_i V \quad , \quad i=1, \ldots, d.
\]
Consequently, for all $d$,
\[
\Wmax{}(\ol{\bD}^d) \subseteq 2d \Wmin{}(\ol{\bD}^d) .
\]
\end{corollary}
\begin{proof}
Put
\[
X= (\re(A_1),\im(A_1),  \ldots, \re(A_d), \im(A_d)) .
\]
Then
\[\cW_1(X) \subseteq \{(\alpha_1,\alpha_2,\ldots,\alpha_{2d} ): |\alpha_{2j-1}+i\alpha_{2j}|\leq 1, 1\leq j \leq d\} = \ol{\mathbb{D}}^d \subseteq \mathbb{R}^{2d}
\]
where $\mathbb{D}$ is the unit disc in $\mathbb{R}^2$.

For every $1\leq m \leq 2d$, write $e_m$ for the $m$-th element of the standard basis of $\mathbb{R}^{2d}$.
Then $e_m e_m^*$ is the projection onto $\mathbb{R}e_m$.
For every such $m$ write $\lambda^{(m)}=2d e_m e_m^*$. Then  $\lambda$ is a $2d$-tuple of real $2d\times 2d$ rank one matrices
such that $I_{2d} \in \conv \{\lambda^{(1)},\ldots,\lambda^{(2d)}\}$.
It is easy to check that $\lambda^{(m)}\cW_1(X) \subseteq 2d\ol{\mathbb{D}}^d$ for every $1\leq m \leq 2d$.
By the previous corollary, $X$ can be dilated to a $2d$-tuple of commuting self-adjoint operators $Y=(Y_1,\ldots, Y_{2d} )$ with $\sigma(Y)\subseteq 2d\ol{\mathbb{D}}^d$.

We now define $T_j:=Y_{2j-1}+iY_{2j}$ for all $1\leq j \leq d$, and it remains to show that $||T_j||\leq 2d$.
For this, note that, since $\sigma(Y)\subseteq 2d\ol{\mathbb{D}}^d $, we can write $Y_m$ as a diagonal matrix $\diag(\alpha_k^{(m)})$ such that, for every $k$, $(\alpha^{(1)}_k,\ldots,\alpha^{(2d)}_k)\in 2d \ol{\mathbb{D}}^d$.
But, then, for every $1\leq j \leq d$, $(\alpha^{(2j-1)}_k,\alpha^{(2j)}_k)\in  2d \ol{\mathbb{D}}$.
Since $T_j=\diag(\alpha^{(2j-2)}_k+i\alpha^{(2j)}_k)$,  $||T_j||\leq 2d$.
\end{proof}

As a Corollary to Theorems \ref{thm:scaled-inclusion-min} and \ref{general_dilation} and Corollary \ref{cor:lambda-over-C-dilation} we get

\begin{theorem}\label{relaxation and symmetry}
Let $\cS\subseteq \cup_n (M_n)^d_{sa}$ and $\cT\subseteq \cup_n (M_n)^d_{sa}$ be  matrix convex sets.
Assume that there is a $k$-tuple of real $d\times d$ rank one matrices $\lambda$
such that $I_d \in \conv \{\lambda^{(1)},\ldots,\lambda^{(k)}\}$ and such that $\lambda^{(m)}\cS\subseteq \cT$ for all $m$ .
Then
\[
\cS \subseteq  \Wmin{}(\cT_1) .
\]
In particular, if there is a constant $C$ and a $k$-tuple of real $d\times d$ rank one matrices $\lambda$
as above, such that $\cS$ is $\frac{1}{C}\lambda$-symmetric, then
\[
\cS \subseteq C \Wmin{}(\cS_1) .
\]
\end{theorem}


\begin{remark}\label{rem:getHKMS}
\cite[Proposition 14.1]{HKMS15} is the very special case of the second part of the above theorem, where $C=d$ and $\frac{1}{d}\lambda^{(m)}$ is the projection onto the $m$th coordinate.
Theorem \ref{relaxation and symmetry} leads to the following obvious extension: if there is an orthonormal  basis $\{e_1, \ldots, e_m\}$ such that a convex set $K \subseteq \bR^d$ is invariant under projections onto these basis vectors, then one can define $\lambda^{(m)}  = d e_m e_m^*$ for $m=1, \ldots, d$, and we see that the assumptions of the second part of Theorem \ref{relaxation and symmetry} hold with $C=d$.
\end{remark}

\begin{remark}\label{rem:getprev}
We now see how Theorem \ref{thm:1dilation} also fits as a special case into the framework of this subsection.

The dilation constructed in Theorem \ref{thm:1dilation} for a tuple $X$ is given by $T_i = \sum_j X_j \otimes W_i W_j$, where $W_i$ is a certain flip operator.
This is a special case of the dilation constructed in Theorem \ref{general_dilation}, with $\cT = d \cdot \fC^{(d)}$ and $S_{ij} = W_i W_j$.
Now, we have $d^2$ operators $S_{ij}$, and if we jointly diagonalize them and denote $S_{ij} = \diag (\lambda_{ij}^{(1)}, \ldots, \lambda_{ij}^{(2^{d-1})})$, then we obtain $k = 2^{d-1}$ real matrices $\lambda^{(m)}$, $1\leq m \leq k$.
Working out the joint eigenvectors of the matrices $W_i W_j$, we find that $\lambda^{(m)}$ is a rank one operator, and that $\frac{1}{d} \lambda^{(m)}$ is the projection onto a one dimensional space in the direction of a vector consisting of $\pm 1$'s.
Moreover, the identity $I_d$ is in $\conv\{\lambda^{(1)}, \ldots, \lambda^{(k)} \}$.
\end{remark}

Let us say that an nc set $\cS$ is {\em symmetric} if $(X_1, \ldots, X_d) \in \cS$ implies $(\epsilon_1 X_{1}, \ldots, \epsilon_d X_{d}) \in \cS$ for every $\epsilon_1, \ldots, \epsilon_d \in \{-1,1\}$.
If $\cS$ is symmetric, it is said to be {\em fully symmetric} if for such $X$ also $( X_{\sigma(1)}, \ldots,  X_{\sigma(d)}) \in \cD$ for every permutation $\sigma$.
Every fully symmetric matrix convex set is invariant under the projections $e_i e_i^*$, where $\{e_1,...,e_d\}$ is the standard orthonormal basis. Hence, by defining $\lambda^{(m)} = d e_i e_i^*$, we can apply the second part of Theorem \ref{relaxation and symmetry} and Corollary \ref{cor:lambda-over-C-dilation} with constant $C=d$.
On the other hand, the fact that not all sets invariant under such projections are fully symmetric (see Example \ref{ex:symmetry} below) shows that Theorem \ref{relaxation and symmetry} and Corollary \ref{cor:lambda-over-C-dilation} have a wider applicability.

\begin{example}\label{ex:symmetry}
Let $d=k=2$ and
\[
\lambda^{(1)}=\left(\begin{array}{cc} 1 & 1 \\ 1 & 1 \end{array}\right)
\qand
\lambda^{(2)}=\left(\begin{array}{cc} 1 & -1 \\ -1 & 1 \end{array}\right) .
\]
If $\mathcal{D}^{sa}_A$ is fully symmetric, i.e., it satisfies $C_i\mathcal{D}^{sa}_A\subseteq \mathcal{D}^{sa}_A$ for
\[
C_1=\left(\begin{array}{cc} 0 & 1 \\ 1 & 0 \end{array}\right), \quad
C_2=\left(\begin{array}{cc} -1 & 0 \\ 0 & 1 \end{array}\right) \qand
C_3=\left(\begin{array}{cc} 1 & 0 \\ 0 & -1 \end{array}\right),
\]
then it is also $\frac{1}{2}\lambda$-symmetric because $\mathcal{D}^{sa}_A$ is convex and
\[
\frac{1}{2}\lambda^{(1)}=\frac{1}{2}(I + C_1) \qand \frac{1}{2}\lambda^{(2)}=\frac{1}{2}(I + C_1C_2C_3) .
\]

The converse is false as we now show.
To do this, look at
\[\cS = \{X=(X_1, X_2): -I\leq X_1-X_2 \leq I,  \;\; -I\leq X_1+X_2 \leq 2I \} .
\]
It is easy to verify that the matrix convex set $\cS$ is $\frac{1}{2}\lambda$-symmetric,
but not fully symmetric (e.g. $(I,I) \in \cS$ but $-(I,I) \notin \cS$).
\end{example}

\begin{example}\label{examplesketch}
Fix $d=2$. In $\mathbb{R}^2$ draw two straight lines that pass through the origin and are not parallel.
Call them $L_1, L_2$.
On each line draw two points (different from the origin), say $T_1,T_3$ on $L_1$ and $T_2,T_4$ on $L_2$ such that the origin lies in the intervals $[T_1,T_3]$ and $[T_2,T_4]$. Write
\[ P_{m}=\conv\{T_i:1\leq i \leq 4\} . \]
Now, through $T_1$ and $T_3$ (on $L_1$) draw straight lines parallel to $L_2$. Similarly, through $T_2$ and $T_4$ draw lines parallel to $L_1$. These $4$ lines form a parallelogram, call it $P_{M}$.
Clearly, $P_{m}\subseteq P_{M}$. Both $P_{m}$ and $P_{M}$ are given by four linear inequalities and, in a natural way, define free spectrahedra $\mathcal{D}^{sa}_{A_m}$ and $\mathcal{D}^{sa}_{A_M}$.

Write $q_i$ for the projection of $\mathbb{R}^2$ onto $L_i$ (parallel to the other line) and set $\lambda^{(i)}=2q_i$, $i=1,2$.
Note that both are real rank one matrices and $q_1+q_2=I_2$.
If $E$ is any set between $P_M$ and $P_m$, then $E$ is $\frac{1}{2}\lambda$ invariant, because each $q_i$ maps $P_M$ into $P_m$.
In fact, we will explain in the next subsection how to apply Theorem \ref{relaxation and symmetry} and Corollary \ref{cor:lambda-over-C-dilation} to any matrix convex set sandwiched between $\mathcal{D}^{sa}_{A_m}$ and $\mathcal{D}^{sa}_{A_M}$.
\end{example}

\subsection{Relaxation theorems and connections to maximal spectrahedra}\label{subsec:relaxationthms}

We now strengthen Theorem \ref{relaxation and symmetry} by weakening the symmetry requirement for $\cS$ to be a requirement for the first level $\cS_1$ only.

\begin{theorem}\label{thm:mainrelax}
Let $\cS$ be a matrix convex set in $\cup_n (M_n)^d_{sa}$.
Assume that there is a constant $C$ and that there are $k$ rank one real $d\times d$ matrices $\lambda :=\{\lambda^{(m)}: 1\leq m \leq k \}$ such that $I_d \in \conv \{\lambda^{(1)},\ldots,\lambda^{(k)}\}$ and such that $\cS_1$ is $\frac{1}{C}\lambda$-symmetric in the sense that $\lambda^{(m)} \cS_1 \subseteq C \cS_1$ for all $1\leq m \leq k$.
Then for every other matrix convex set $\cT$, we have
\[
\cS_1 \subseteq \cT_1 \Longrightarrow \cS \subseteq C \cT .
\]
\end{theorem}

\begin{proof}
If $\cS_1$ is $\frac{1}{C}\lambda$-symmetric, then so is $\Wmax{} (\cS_1)$, since it is defined by the same linear inequalities (see Section \ref{sec:maxmin}).
By Theorem \ref{relaxation and symmetry},
\[
\Wmax{}(\cS_1) \subseteq C \cT .
\]
By maximality, $\cS \subseteq \Wmax{}(\cS_1) \subseteq C \cT$.
\end{proof}

We now obtain a spectrahedral inclusion theorem in the spirit of \cite{HKMS15}.
It is interesting to compare the following corollary with \cite[Proposition 8.1]{HKMS15}, which has a similar bound, but one which depends on the ranks of the matrices, not on their number.
The sharpness of the constant $d$ in the following Corollary will be obtained in Example \ref{ex:dsharp}.

\begin{corollary}\label{cor:mainrelax}
Let $A$ be a $d$-tuple of self-adjoint operators, and assume that $\cD^{sa}_A(1)$ is invariant under projection onto some orthonormal basis $($see Remark $\ref{rem:getHKMS}\,)$.
Then for every $d$-tuple of self-adjoint operators $B$, we have
\[
\cD^{sa}_A(1) \subseteq \cD^{sa}_B(1) \Longrightarrow \cD^{sa}_A \subseteq d \cD^{sa}_B.
\]
Moreover, the constant $d$ is sharp, in the sense that for every $d$ there is a $d$-tuple $A$ such that $\cD^{sa}_A(1)$ is fully-symmetric, but $\cD^{sa}_A \nsubseteq C \Wmin{}(\cD^{sa}_A(1))$ for every $C < d$.
\end{corollary}

\begin{corollary}\label{cor:mainrelax2}
Suppose that $\cS$ is a matrix convex set as in Theorem \ref{thm:mainrelax}.
Then
\[
\Wmin{}(\cS_1) \subseteq \cS \subseteq \Wmax{}(\cS_1) \subseteq C \Wmin{}(\cS_1).
\]
\end{corollary}

\begin{example}[The regular simplex]\label{ex:simplex}
Here is an example of a convex set in $\bR^d$ for which the assumptions of Corollary \ref{cor:mainrelax} do not hold, but to which we may still apply Theorem \ref{thm:mainrelax} and Corollary \ref{cor:mainrelax2}.
The {\em 3-simplex} $\Delta^3$ in $\bR^3$ is the convex hull of $v_1 = (1,1,1)$, $v_2 = (1,-1,-1)$, $v_3 = (-1,1,-1)$ and $v_4 = (-1,-1,1)$.
It is not invariant under projection onto any orthonormal basis.

Put $\lambda^{(m)} = v_m v_m^*$, $m=1, \ldots, 4$.
Then $\Delta^3$ is $\frac{1}{3} \lambda$ invariant, since $\frac{1}{3}\lambda^{(m)}$ is the orthogonal projection onto $v_m$, and one can actually {\em see} that $\Delta^3$ is invariant under that.
One computes directly $\sum_{m=1}^4  \lambda^{(m)} = 4I$, thus $I \in \conv\{\lambda^{(m)} : 1\leq m \leq 4\}$ so Theorem \ref{thm:mainrelax} and Corollary \ref{cor:mainrelax2} are applicable with $C = 3 = d$.
\end{example}

\begin{example}[The right angled simplex]\label{ex:rasimplex}
The point of this example is to show that $\frac{1}{C}\lambda$ symmetric sets need not be symmetrically situated about the origin.
Consider the convex set $K = \conv\{0,e_1, e_2, e_3\} \subset \bR^3$, where $e_1, e_2, e_3$ denote the standard basis.
Putting $\lambda^{(i)} = 3 e_i e_i^*$, we find that $K$ is $\frac{1}{3} \lambda$ symmetric and $I =  \sum_{i=1}^3 \frac{1}{3} \lambda^{(i)}$.
Thus Corollary \ref{cor:mainrelax2} applies, so
\[\Wmax{}(K) \subseteq 3 \Wmin{}(K). \]
\end{example}

\begin{remark}
The last part of Corollary \ref{cor:mainrelax} says that $C = d$ is the smallest constant that works for all fully symmetric matrix convex sets.
However, we do not know whether a particular $d$-dimensional convex set $K$ exists such that $\Wmax{}(K) \subseteq C\Wmin{}(K)$ for some $C<d$.
In particular, we do not know whether $d$ is the optimal constant for $K = [-1,1]^d$.
On the other hand, it is easy to construct non-symmetric examples where the best constant that we can find is bigger than $d$.
If $K$ is convex set with $0$ in the interior, let $c$ be a positive number such that $cK \subseteq \tilde{K}$, where $\tilde{K} \subseteq K$ is a convex set which is $\frac{1}{d} \lambda$-invariant for $\lambda$ as in Theorem \ref{thm:mainrelax}.
Then
\[
\Wmax{}(cK) \subseteq \Wmax{}(\tilde{K}) \subseteq d\Wmin{}(\tilde{K}) \subseteq d \Wmin{}(K),
\]
Thus $\Wmax{}(K) \subseteq \frac{d}{c} \Wmin{}(K)$.
This conclusion could also have been obtained by noting that if $\tilde{K}$ is $\frac{1}{d}\lambda$ invariant, then $K$ is $\frac{c}{d}\lambda$ invariant.

Although we do not have anything to say regarding the optimality of the constant $\frac{d}{c}$, the fact that it blows up as $c \to 0$ is consistent with Example \ref{ex:non-scalable} below.
\end{remark}

\subsection{A non-scalable example}

\begin{example} [A non-scalable inclusion]\label{ex:non-scalable}

Let $T = \left(\begin{smallmatrix} 1 & 2 \\ 0 & 1 \end{smallmatrix}\right)$.
Then
\[
\cW_1(T) = \ol{\bD}_1(1) = \{ \ z \in \bC \ | \ |z-1| \leq 1 \ \},
\]
is a disc containing $0$ on the boundary. Let $N = M_{1+z}$ be the multiplication operator on $L^2(\bT)$, so that $\sigma(N) = 1 + \bT$.
We see that $\cW_1(T) = \cW_1(N)$.
By Corollary \ref{cor:normal} we have $\cW(N) = \Wmin{}(\ol{\bD}_1(1))$.
Thus $\cW(T) \supseteq \cW(N)$.
However, there is no $C$ such that $\cW(T) \subseteq C\cW(N)$.
To see this, we show that there is no UCP map $\phi : S_N \rightarrow S_T$ sending $N$ to $cT$ for any $c > 0$, and invoke Theorem \ref{thm:ExistUCP_matRan}.

Indeed, if there were such a map, let $U = M_z = N-1$, then $\phi(U) = cT-I$. But $|| cT - I || > 1$ so this is impossible.
Indeed, observe that
\[
\| cT - I \| = \left\| \begin{pmatrix} c-1 & 2c \\ 0 & c-1 \end{pmatrix} \right\| = \left\| \begin{pmatrix} 2c & 1-c \\ 1-c & 0 \end{pmatrix} \right\| .
\]
But this is equal to the largest root of $t^2 - 2ct - (1-c)^2 = 0$.
Substituting $u=t-1$ to get $u^2 + 2(1-c)u - c^2 = 0$, this equation must now have a positive root, so that $\|cT - I \| > 1$.

To get an example involving free spectrahedra, we apply the polar dual and check that
\[
\cD_T(1) = (\cW(T)(1))' = (\cW(N)(1))' = \cD_N(1) ,
\]
while there is no constant $C$ such that $\cD_N \subseteq C \cD_T$.
\end{example}

\subsection{Optimality of the constant $C = d$}
The following lemma is likely well-known, but we do not have a convenient reference.

\begin{lemma} \label{lem:example}
For every $d$, there exist $d$ self-adjoint $2^{d-1} \times 2^{d-1}$ matrices $B_1, \ldots, B_d$ such that for all $v \in \bR^d$, $\|v\|=1$,
\[
\sum v_i B_ i \leq I,
\]
and such that $d$ is an eigenvalue of $\sum_{i=1}^d B_i \otimes B_i$.
Hence if $\sum_{i=1}^d B_i \otimes B_i \leq \rho I$, then $\rho \geq d$.
\end{lemma}

\begin{proof}
The proof is by induction.
For $d = 1$ we take $B_1 = [1]$.
Suppose that $d \geq 1$, and let $B_1, \ldots, B_d$ be self-adjoint $2^{d-1} \times 2^{d-1}$ matrices as in the statement of the lemma.
We will construct  self-adjoint $2^d\times 2^{d}$ matrices $B'_1, \ldots, B'_{d+1}$ as required.

Let
\[
E_1 = \begin{pmatrix} 0 & 1 \\ 1 & 0\end{pmatrix} \,\, , \,\, E_2 =  \begin{pmatrix} 1 & 0 \\ 0 & -1\end{pmatrix}.
\]
Define
\[
B'_i = E_1 \otimes B_i \,\, , \,\, \textrm{ for } i=1, \ldots, d,
\]
and
\[
B'_{d+1} = E_2 \otimes I_{2^{d-1}} .
\]
The matrices $B'_1, \ldots, B'_{d+1}$ are self-adjoint.
For a unit vector $(v_1,\dots,v_{d+1})$, we compute
\[
I - \sum_{i=1}^{d+1} v_i B'_i =
\begin{pmatrix}
(1 - v_{d+1}) I & -\sum_{i=1}^d v_i B_i \\
-\sum_{i=1}^d v_i B_i & (1 + v_{d+1}) I  ץ
\end{pmatrix}
\]
By \cite[Lemma 3.1]{PauBook}, this matrix is positive semidefinite if and only if
\[
\big( \sum_{i=1}^d v_i B_i \big)^2 \le (1-v_{d+1}^2) I .
\]
By the inductive hypothesis,
\[
\big( \sum_{i=1}^d v_i B_i \big)^2 \leq \big( \sum_{i=1}^d v_i^2 \big)I \leq (1 - v_{d+1}^2)I .
\]
Therefore $\sum_{i=1}^{d+1} v_i B'_i \le I$.

It remains to show that $d+1$ is an eigenvalue of $\sum_{i=1}^{d+1} B'_i \otimes B'_i$.
We write $T = \sum_{i=1}^d B_i \otimes B_i$ and examine the operator
\[
\sum_{i=1}^{d+1} B'_i \otimes B'_i \simeq
\begin{pmatrix}
I  & 0  & 0  &  T\\
0 & - I & T & 0 \\
0 & T & -I & 0  \\
T & 0 & 0  & I
\end{pmatrix}.
\]
Now, if $x$ is an eigenvector of $T$ corresponding to $d$, then $(x, 0, 0, x)^t$ is an eigenvector of $\sum_{i=1}^{d+1} B'_i \otimes B'_i$ corresponding to the eigenvalue $d+1$.
\end{proof}

\begin{example}\label{ex:dsharp}
We construct tuples of operators $A$ and $B$ of self-adjoint matrices, such that $\cD^{sa}_A$ is symmetric, and such that the implication in Corollary \ref{cor:mainrelax} holds with a constant $d$ but with no smaller constant.
In this example the tuple $A$ consists of operators on an infinite dimensional space, but it is not hard to see that this implies sharpness in the finite dimensional case as well.

Let $\{v^{(n)}\}$ be a dense sequence of points on the unit sphere of $\bR^d$.
Let $A$ be the $d$-tuple of diagonal operators such that $n$th element on the diagonal of $A_j$ is the $j$th coordinate $v^{(n)}_j$ of $v^{(n)}$.
Then
\[
\cD^{sa}_A = \{X  \in  M_n(\bC)^d_{sa} : \sum X_j v_j \leq I \FORAL v \in \bR^d,\ \|v\| = 1 \} .
\]
Observe that $\cD^{sa}_A = \Wmax{}(\ol{\bB}_d)$, and in particular $\cD^{sa}_A(1)$ is the unit ball of $\bR^d$ (which is fully-symmetric and invariant under projections onto any orthonormal basis).

Let $B$ be as in Lemma \ref{lem:example}.
For every unit vector $v \in \bR^d$, $\sum v_i B_i \leq I$.
Thus $\cD^{sa}_A(1) \subseteq \cD^{sa}_B(1)$; and moreover, $B \in \cD^{sa}_A$.
On the other hand, $B \notin C \cD^{sa}_B $ for any $C < d$, since $ \sum_i B_i \otimes B_i$ has an eigenvalue equal to $d$.
Thus $\cD^{sa}_A \nsubseteq C \cD^{sa}_B$ for any $C < d$.
\end{example}

\section{Examples of \texorpdfstring{$\frac{1}{C}\lambda$} \ -symmetric polytopes arising from frames}\label{sec:examples}

In general, given a convex set $K \subseteq \bR^d$, and an orthonormal basis $\cB = \{e_1,...,e_d\}$,
it is easy to check if $K$ is invariant under the projections $e_ie_i^*$ for $1 \le i \le d$.
But it is harder to actually find or disprove the existence of an orthonormal basis that leaves $K$ invariant.
Instead of this, we turn to tight frames, which are often used to \emph{define} many convex polytopes $K \subset \bR^d$.
We will show that the convex polytopes generated by a large class of tight frames satisfy the symmetry conditions of Theorems \ref{cor:lambda-over-C-dilation} and \ref{relaxation and symmetry}
This context will encompass all the symmetric situations that we have dealt with so far.

A set of unit vectors $\Phi:= \{v_1,...,v_N\} \subseteq \bR^d$ (without repetition) is called a {\em tight frame} if there is a constant $\sigma>0$ such that for all $x\in \bR^d$ we have
$$
\sum_i |\langle x,v_i \rangle|^2 = \sigma \| x \|^2 ;
$$
this condition is equivalent to $\sum_i v_i v_i^* = \sigma I$.

When all the vectors $v_i$ are of the same length $\ell$, we call $\Phi$ an \emph{isometric} tight frame, and it turns out that in this case we have $\sigma = \ell^2 \cdot \frac{N}{d}$ (See \cite[Lemma 2.1]{RW02}).

Every tight frame $\Phi$ gives rise to a finite subgroup of isometric symmetries given by
$$
\Sym(\Phi) := \{U\in \cU(\bR^d) | U\Phi = \Phi \},
$$
where $\cU(\bR^d)$ denotes the unitary group on $\bR^d$.
We can then turn this construction around, and define isometric tight frames from finite subgroups of $\cU(\bR^d)$.
This will provide us with an abundance of examples.
By \cite[Theorem 6.3]{VW05}, for a finite irreducible subgroup $G$ and a non-zero vector $\phi \in \bR^d$, the set $\Phi = \{g\phi \}_{g\in G}$ is an isometric tight frame, and $K = \conv \Phi$ is invariant under $G$.

\begin{proposition} \label{P:ITF-inv}
Let $K \subseteq \bR^d$ be a closed convex set, and $\{v_1,...,v_N\}$ be an isometric tight frame in $\bR^d$
with vectors of length $\ell$ such that $\frac{1}{\ell^2} v_iv_i^*(K) \subseteq K$.
Then $\Wmax{}(K) \subseteq d \Wmin{}(K)$.
\end{proposition}

\begin{proof}
Up to normalization, we may assume $\ell=1$. By prescribing $\lambda^{(i)} = d \cdot v_iv_i^*$ and using $\sigma = \frac{N}{d}$, we have that
$$
\sum_i \lambda^{(i)} = \sum_i d v_i v_i^* = d \sigma \cdot I = N \cdot I
$$
so that
\[
I = \frac{1}{N} \sum_{i=1}^N \lambda^{(i)} \in \conv \{\lambda^{(i)}\} \qand
\frac{1}{d}\lambda^{(i)}(K) = v_i v_i^* (K) \subset K .
\]
Therefore $K$ is $\frac{1}{d}\lambda$-symmetric.
By Theorem \ref{thm:mainrelax}, we see that
\[
\Wmax{}(K) \subseteq d \Wmin{}(K) .\qedhere
\]
\end{proof}

Our goal in the remainder of this subsection is to find classes of tight frames for which the conditions of Proposition \ref{P:ITF-inv} hold with $K = \conv \Phi$.
We begin with the following simple condition.

\begin{corollary} \label{C:SITF}
Let $\Phi = \{v_1,...,v_N\} \subset \bR^d$ be an isometric tight frame with vectors of length $\ell$ and $K = \conv \Phi$.
If $- \Phi = \Phi$, then
\[ \Wmax{}(K) \subseteq d \Wmin{}(K) .\]
\end{corollary}

\begin{proof}
Up to normalization, we may assume $\ell=1$. We need only verify that $v_iv_i^*(v_j) \in K$ for every $1 \leq i,j \leq N$. But then
\[
v_iv_i^*(v_j) = \langle v_j, v_i \rangle v_i
\]
Now if $\big\langle v_j, v_i \big\rangle \geq 0$ then $v_iv_i^*(v_j)$ is simply a rescaling of $v_i$ by a constant $0 \leq c \leq 1$ and is hence in $\conv \Phi$.
If $\langle v_j, v_i \rangle \leq 0$, then $v_iv_i^*(v_j)$ is a rescaling of $-v_i$ by a constant $0 \leq c \leq 1$ and is hence in $\conv \Phi$.
In either case we have that $v_iv_i^*(v_j) \in K$ so that by Proposition \ref{P:ITF-inv} we have
\[ \Wmax{}(K) \subseteq d \Wmin{}(K) . \qedhere \]
\end{proof}

The assumption $\Phi = - \Phi$ is rather restrictive (consider Example \ref{ex:simplex}).
For the purpose of exhibiting a class of isometric tight frames for which the invariance condition $\frac{1}{\ell^2} v_iv_i^*(K) \subseteq K$ in Proposition \ref{P:ITF-inv} is automatic, we bring forth the following definition.
For a tight frame $\Phi$, denote by
\[ \Stab(v) = \{ U \in \Sym(\Phi) : Uv=v \} ,\]
the stabilizer subgroup of $\Sym(\Phi)$ of symmetries that fix the vector $v\in \mathbb{R}^d$.

\begin{definition}
Let $\Phi = \{v_1,...,v_N\}$ be a tight frame.
We say that $\Phi$ is \emph{vertex reflexive} if $\Stab(v_i)$ fixes a subspace of dimension exactly one, namely $\spn\{v_i\}$, for every $1 \leq i \leq N$.
\end{definition}

If $\Phi = \{v_1,..., v_N\}$ is an isometric tight frame, then no $v_i$ can be a convex combination of $\Phi \setminus \{v_i\}$.
This means that the vectors $\{v_1,...,v_N\}$ comprise the vertices of a $d$-dimensional polytope.

Every element of $\Sym(\Phi)$ must leave the barycenter $\frac{1}{N}\sum_{i=1}^N v_i$ invariant, thus when $\Phi$ is a vertex reflexive isometric tight frame we must have that $\frac{1}{N}\sum_{i=1}^N v_i = 0$.
Hence, $0$ must be an interior point of $K = \conv(\Phi)$, since the barycenter of every $d$-dimensional convex set in $\bR^d$ is in the interior.

We say that a face $F$ of a polytope $K$ is {\em $m$-dimensional}, if $m$ is the minimal dimension of an affine subspace containing $F$. For an isometric tight frame $\Phi$, we must have that every element of $\Sym(\Phi)$ maps $m$-dimensional faces to $m$-dimensional faces.

For a computational method for constructing many examples of vertex reflexive isometric tight frames (satisfying the additional requirement that $\Sym(\Phi)$ is irreducible and transitive), see \cite{BW13}.

\begin{proposition}\label{prop:VRITF}
Let $\Phi = \{v_1,...,v_N\}$ be a vertex reflexive isometric tight frame of vectors of length $\ell$, and let $K:= \conv(\Phi)$ be the $d$-dimensional convex polytope generated by $\Phi$.
Then $\frac{1}{\ell^2}v_i v_i^*(K) \subseteq K$ for all $1 \leq i \leq N$.
\end{proposition}

\begin{proof}
Up to normalization, we may assume $\ell=1$. Fix $1\leq i \leq N$. Let $\alpha > 0 $ be maximal such that $-\alpha v_i \in K$, and let $F$ be a face of $K$ of minimal dimension $m$ such that $-\alpha v_i \in F$.

We first claim that every element $g\in \Stab(v_i)$ must leave $F$ invariant.
Indeed, if not, $g(F)$ must be an $m$-dimensional face with $-\alpha v_i \in g(F)$ which is different from $F$.
Since $g(F) \cap F$ must be a face of dimension strictly less than $m$, we arrive at a contradiction to the definition of $F$.

Since $F = \conv \{v_{i_1},...,v_{i_p}\}$ is left invariant under $\Stab(v_i)$, we may restrict each element $g\in \Stab(v_i)$ to the subspace $W = \spn( \{v_i\} \cup F)$.
Within $W$, since every $g\in \Stab(v_i)$ maps $F$ to itself, it must then map the affine subspace $A$ generated by $F$ inside $W$, to itself, and hence must map the normal of $A$ (again inside $W$) to itself.
But even within the subspace $W$, we still have that $\Stab(v_i)$ fixes a subspace of dimension exactly one, and hence the normal of $A$ in $W$ can be chosen to be $v_i$. In other words, $v_i$ is perpendicular to $-\alpha v_i - v_{i_j}$ for any $1 \leq j \leq p$.
This means that
\[
v_iv_i^*(v_{i_j}) = \langle v_i, v_{i_j} \rangle v_i = \langle v_i, -\alpha v_i \rangle v_i = -\alpha v_i \in K
\]
and $- \alpha = \langle v_i , v_{i_j} \rangle$ is the cosine of the angle between $v_i$ and $v_{i_j}$ for all $1 \leq j \leq p$. We note that by maximality of $\alpha$, we have for each $1 \leq k \leq N$ that the angle between $v_i$ and $v_k$ is at most $\arccos(- \alpha)$.

Hence,  for all $1 \leq k \leq N$ we have that $v_iv_i^*(v_k) = \langle v_i, v_k \rangle v_i$ is a convex combination of $-\alpha v_i$ and $v_i$ and is hence in $K$. As $v_iv_i^*$ is linear, we see that $v_iv_i^*(K) \subseteq K$ as required.
\end{proof}

Combining Propositions \ref{P:ITF-inv} and \ref{prop:VRITF}, we obtain the following.

\begin{theorem} \label{T:HSITF}
Let $\Phi$ be a vertex reflexive isometric tight frame, and let $K = \conv(\Phi)$ be the convex polytope it generates.
Then
\[ \Wmax{}(K) \subseteq d \Wmin{}(K) . \]
\end{theorem}

As a consequence of Theorem \ref{T:HSITF} and \cite[Theorem 5.4]{BW13}, we obtain our results for any convex regular polytope (See \cite[Definition 5.1]{BW13}).

\begin{corollary} \label{C:CRP-dil}
Let $K = \conv\{v_1,...,v_N \}$ be a convex regular \emph{real} polytope according to \cite[Definition 5.3]{BW13}.
Then $\Wmax{}(K) \subseteq d \Wmin{}(K)$.
\end{corollary}

\begin{example}
In \cite{BW13}, a class of frames called {\em highly symmetric} frames was studied.
A {\em highly symmetric tight frame} is a vertex reflexive tight frame for which $\Sym(\Phi)$ is also transitive and irreducible (and is then automatically isometric).
The class of highly symmetric frames was shown to be rich, yet tractable.

We will now construct an example of a vertex reflexive isometric tight frame $\Theta$ for which $\Sym(\Theta)$ is not irreducible, not transitive, and for which \emph{no} vector $u\in \Theta$ satisfies $-u\in \Theta$.
Thus Theorem \ref{T:HSITF} applies, while Corollaries \ref{C:SITF} and \ref{C:CRP-dil} do not.

Let $G = S_5$ act on $(e_1+...+e_5)^{\perp}$ inside $\mathbb{R}^5$, where $\{e_1,...,e_5\}$ is the standard orthonormal basis, and $S_5$ acts by permutation matrices. Take the vector $\phi:= 3w_2 = (3,3,-2,-2,-2)$. Then by \cite[Example 4]{BW13} the frame $\Phi_2:= (g \phi)_{g \in S_5}$ is a vertex-reflexive isometric tight frame comprised of $10$ distinct vectors, and by construction we see that for all $v\in \Phi_2$, we have $-v \notin \Phi_2$.

Hence, let $\Phi = \{v_1,...,v_{10}\}$ be a unit-norm vertex-reflexive tight frame in $\mathbb{R}^4$, when we identify $\Phi$ inside $\mathbb{R}^4$ with the normalization of $\Phi_2$ inside $(e_1+...+e_5)^{\perp} \subseteq \mathbb{R}^5$. So we still have $-v\notin \Phi$ for all $v\in \Phi$.

We then take the vertex-reflexive unit-norm tight frame of the pentagon inside $\mathbb{R}^2$,
$$
\Psi: = \{ (\cos(2\pi k / 5),\sin(2 \pi k / 5)) \}_{k=1}^5
$$
Which has $5$ distinct elements, and satisfies $-w \notin \Psi$ for all $w \in \Psi$. Then define
$$
\Theta = (\Phi \oplus 0_2) \cup (0_4 \oplus \Psi)
$$
Which satisfies $u\notin \Theta$ for all $u\in \Theta$. We know by \cite[Lemma 2.1]{RW02} that
$$
\sum_{v \in \Phi} vv^* = \frac{10}{4} P \ \ \text{and} \ \ \sum_{w \in \Psi} ww^* = \frac{5}{2} Q
$$
Where $P$ and $Q$ are the orthogonal projections onto $\mathbb{R}^4$ and $\mathbb{R}^2$ respectively, that sum to the identity $I$ on $\mathbb{R}^6 = \mathbb{R}^4 \oplus \mathbb{R}^2$. Therefore,
$$
\sum_{u\in \Theta} uu^* = \frac{5}{2} (P+Q) = \frac{5}{2} I .
$$
Thus $\Theta$ is a unit-norm tight frame in $\mathbb{R}^6$. Since $\Sym(\Phi)$ and $\Sym(\Psi)$ fix $\mathbb{R}^4$ and $\mathbb{R}^2$ in the decomposition of $\mathbb{R}^6$ above, and $\Phi$ and $\Psi$ are vertex-reflexive, in their respective spaces, we have that their union $\Theta$ is also a vertex-reflexive unit-norm tight frame.

Since elements of $\Phi$ and $\Psi$ are always perpendicular when identified as elements of $\mathbb{R}^6$, no symmetry of $\Theta$ can map an element of $\Phi$ to an element of $\Psi$. Indeed, from the construction of $\Phi_2$, an element $v_i$ from $\Phi$ has no element perpendicular to it from $\Phi$, so that in $\Theta$, there are only $5$ elements perpendicular to $v_i$: those of $\Psi$. On the other hand, an element $w_j$ of $\Psi$ has exactly $10$ elements of $\Theta$ perpendicular to it: those of $\Phi$. Thus, no $v_i \in \Phi$ can be mapped to any $w_j \in \Psi$ via by a symmetry of $\Theta$, and $\Sym(\Theta)$ is not transitive. Hence, elements of $\Sym(\Theta)$ can only permute elements of $\Phi$ among themselves, and elements of $\Psi$ among themselves. Thus
\[ \Sym(\Theta) = \Sym(\Phi) \oplus \Sym(\Psi) ,\]
so that $\Sym(\Theta)$ is reducible.
\end{example}

\section{The matrix ball}\label{sec:ball}

Let $\ol{\bB} = \ol{\bB}_d$ denote the closed unit ball in $\bR^d$.
Recall that the ($d$-dimensional) {\em matrix ball} is defined to be
\[
\fB = \fB^{(d)} = \{X \in  M_n(\bC)^d_{sa}: \sum_{j=1}^d X_j^2 \leq I\}.
\]
We also introduce another ``ball'' which will turn out to conform more naturally to our duality.
Recall that the transpose is a linear map from $A\in\cB(H)$ to $\cB(H^*)$ is given by $A^t f = f \circ A$.
Define the \textit{conjugate} of $A \in \cB(H)$ to be $\ol A := (A^*)^t$.
It is called the conjugate because if $A = [a_{ij}]$ belongs to $M_n$, then $\bar A = [\ol{a_{ij}}]$.
Haagerup \cite[Lemma 2.4]{Haag83} established an important identity for the spatial tensor product by observing that
if $H$ and $K$ are Hilbert spaces, then a spatial tensor product $A \otimes B$ on $H \otimes K$ of two operators $A\in B(H)$ and $B \in B(K)$ can be represented as an
operator on the Hilbert-Schmidt operators $\cS_2(K,H)$ from $K$ into $H$,
which is canonically isomorphic to the Hilbert space $H \otimes K^*$.
Indeed, the operator $A \otimes \ol{B}$ is unitarily equivalent to the operator $u \to AuB^*$ in $\cB(\cS_2(K,H))$.
Haagerup shows that (in the spatial tensor norm)
\[
 \big\| \sum A_i \otimes \ol{B_i} \big\| \le
 \big\| \sum A_i \otimes \ol{A_i} \big\|^{1/2}  \big\| \sum B_i \otimes \ol{B_i} \big\|^{1/2} .
\]

Thus we may define
\[
\fD = \fD^{(d)} = \{ X \in  M_n(\bC)^d_{sa} : \big\| \sum_{j=1}^d X_j \otimes \ol{X_j} \big\| \leq 1 \}.
\]

\begin{lemma}
$\fD$ is a closed matrix convex set.
\end{lemma}

\begin{proof}
Clearly $\fD$ is closed.
Observe that $\fD = -\fD$ and is invariant under the conjugation map sending $A$ to $\ol{A}$.
In particular, the norm condition is equivalent to the two inequalities
\[ \pm \sum_{j=1}^d X_j \otimes \ol{X_j}  \leq  I .\]
Haagerup's inequality immediately shows that if $X\in \fD(m)$ and $Y\in\fD(n)$, then $X \oplus Y \in \fD(m+n)$.
It is also routine to show that if $A \in M_{mn}$ is a contraction, and $X \in \fD(n)$, then $AXA^* \in \fD(m)$.
So $\fD$ is matrix convex.
\end{proof}

Observe that $\fB = - \fB$ and is also closed under conjugation, and conjugation is isometric.
A self-adjoint $d$-tuple $X$ belongs to $\fB$ exactly when
\[ \|X\| = \| [ X_1\ X_2\ \dots\ X_d ] \| = \big\| \sum_{j=1}^d X_j^2\, \big\|^{1/2} \le 1 .\]
Therefore
\begin{align*}
 \big\| \textstyle\sum_{j=1}^d X_j \otimes \ol{X_j} \big\| &\le
 \| X \otimes I \| \, \| I \otimes \ol{X} \|  \\ &=
 \big\| \textstyle\sum_{j=1}^d X_j^2 \otimes I \big\|^{1/2} \ \big\| \textstyle\sum_{j=1}^d I \otimes \ol{X_j}^2 \big\|^{1/2} \\ &=
 \| X\|\, \|\ol{X}\| \le 1.
\end{align*}
It follows that $\fB\subseteq \fD$.

Clearly $\fB(1) = \fD(1) = \ol{\bB}$.
Thus we have
\[
\Wmin{}(\ol{\bB}) \subseteq \fB \subseteq \fD \subseteq \Wmax{}(\ol{\bB}) .
\]
It is natural to ask about the precise place these matrix balls take in this inequality.

\begin{lemma}\label{L:fD}
$\fD$ is self-dual, i.e., $\fD^\bullet = \fD$.
\end{lemma}

\begin{proof}
Note that $\fD^\bullet = -\fD^\bullet$ and is also closed under conjugation.
Haagerup's inequality shows that if $X$ and $Y$ belong to $\fD$, then
\[  \big\| \sum X_i \otimes \ol{Y_i} \big\| \le 1. \]
Since $\ol{Y}$ also belongs to $\fD$, we deduce that $L_X(\pm Y) \ge 0$.
It follows that $\fD^\bullet \supseteq \fD$.

Conversely, suppose that $Y$ belongs to $\fD^\bullet$.
Since $\fB$ has $0$ in its interior, so does $\fD$.
Let
\[ r_0 = \sup \{ r : rY \in \fD \} .\]
If $r_0 \ge 1$, then $Y \in \fD$.
But if $r_0<1$, we have $\pm\ol{Y} \in \fD^\bullet$, so that  $L_{r_0Y}(\pm\ol{Y}) \ge 0$.
This implies that
\[  \big\| \sum r_0 Y_i \otimes \ol{Y_i} \big\| \le 1 .\]
However this clearly means that $\sqrt{r_0}Y$ belongs to $\fD$.
Hence $\sqrt{r_0} \le r_0$ and so $r_0\ge 1$ as desired.
\end{proof}

\begin{remark}\label{R:unique self-dual}
If $\cS$ is a self-dual matrix convex set in $(M_n)^d_{sa}$, then $\cS_1 = \ol{\bB_d}$.
Indeed, for every $x \in \cS_1$, we have that $\langle x, x \rangle \leq 1$, thus $\cS_1 \subseteq \ol{\bB_d}$.
It follows that $\cS \subseteq \Wmax{}(\ol{\bB_d})$, so $\cS = \cS^\bullet \supseteq \Wmin{}(\ol{\bB_d})$ (Theorem \ref{T:polar of min}), and in particular we obtain the reverse inclusion $\cS_1 \supseteq \ol{\bB_d}$.

If $\cS$ is also closed under multiplication by $\pm1$ and conjugation, then $\|\sum X_j \otimes \ol{X}_j \| \leq 1$ for $X \in \cS$, thus $\cS \subseteq \fD$.
Applying the polar dual we find $\cS \supseteq \fD$.
That is, $\fD$ is the unique self-dual matrix convex set closed under  multiplication by $\pm1$ and conjugation.
\end{remark}

For this reason, we will call $\fD$ the \emph{self-dual matrix ball}.
We obtain the following immediate consequence.

\begin{corollary}
\[
\Wmin{}(\ol{\bB}) \subset \fB \subset \fD=\fD^\bullet  \subset \fB^\bullet\subset \Wmax{}(\ol{\bB})
\]
and these containments are all proper for $d>1$.
\end{corollary}

\begin{proof}
 For the first containment,  let
\[
X_1 = \begin{pmatrix}  \frac{1}{2} & 0 \\ 0 & 0 \end{pmatrix} \,\, , \,\, X_2 = \begin{pmatrix}  0 & \frac{3}{4} \\ \frac{3}{4} & 0 \end{pmatrix}.
\]
One verifies $X \in \fB$, but by \cite[Example 3.1]{HKM13}, $X$ is not contained in some other spectrahedron $\cD_\Gamma$ with $\cD_\Gamma(1) = \ol{\bB}$.
In particular, $X$ is not in $\Wmin{}(\ol{\bB})$.
(Actually, $\Gamma = L_E$, where $E = (E_1, E_2)$ is given in Lemma \ref{lem:example}.)

If $d>1$, $\fB = \fB^{(d)}$ is not equal to its own polar dual.
Indeed, for $i=1, \ldots, d$, let $B_i$ be the matrix on $\bC^{d+1}$ that switches between $e_1$ and $e_{i+1}$ and sends all other basis vectors to $0$.
Then $\fB = \cD_B^{sa}$ (recall Example \ref{ex:def_ball}), thus $B \in \fB^\bullet$.
On the other hand,
\[ \sum_j B_j^2 = I + (d-1) e_1 e_1^* , \]
where $e_1 e_1^*$ denotes the orthogonal projection onto the first basis vector.
Therefore, $B$ is not in $\fB$. So $\fB^\bullet \ne \fB$.

As $\fB \subseteq \fD$, which is self-dual by the Lemma~\ref{L:fD}, we have $\fB \subsetneq \fD$.
By duality, we obtain $\fD \subsetneq \fB^\bullet \subsetneq \Wmax{}(\ol{\bB})$.
\end{proof}

\begin{remark}
It is not possible to use $\|\sum A_i \otimes A_i \|$ instead of $\|\sum A_i \otimes \ol{A_i} \|$ in the definition of $\fD$, as numerical computations on $3\times 3$ matrix tuples show that these quantities are different in general.
\end{remark}

By analogy to the matrix cube problem, we ask for which constant $C$ does the following implication hold:
\be\label{eq:MBP}
\ol{\bB} \subseteq \cD_A(1) \Longrightarrow \fD \subseteq C \cD_A .
\ee
Note that this is not in perfect analogy with the matrix cube problem, because $\fC = \Wmax{}([-1,1]^d)$, whereas $\fD$ is somewhere near the `center' of the range of matrix convex sets with first level equal to $\ol{\bB}$.
However, we already completely solved the problem for $\Wmax{}(\ol{\bB})$ above in Corollary \ref{cor:mainrelax}.
We know by Corollary \ref{cor:mainrelax} that $C=d$ works in \eqref{eq:MBP}, but we will do better in this case.

Since $\fD \neq \Wmax{}(\ol{\bB})$, we also ask for a constant $c$ such that
\be\label{eq:MBPreverse}
\cD_A(1) \subseteq \ol{\bB} \Longrightarrow  c \cD_A  \subseteq \fD.
\ee

\begin{remark}
In a recent revision of the paper \cite{HKMS15} (that appeared after we obtained the results of this section), results similar to those in this section were obtained using different methods.
It is worth noting that \cite{HKMS15} treats four sets which they call matrix balls: $\fB^{\max}$, which is what we denote by $\Wmin{}(\ol{\bB})$; $\fB^{\min}$, which is what we denote by $\Wmax{}(\ol{\bB})$; $\fB^{\operatorname{oh}}$, which is what we denote by $\fB$; and finally, $\fB^{\operatorname{spin}}$, which is a certain free spectrahedra with $\fB^{\operatorname{spin}}(1) = \ol{\bB}$ which we do not discuss. (However, the spin matrices have arisen in Example \ref{ex:dsharp}.)
\end{remark}

\begin{theorem}\label{thm:MBP}
Let $\cS \subseteq \cup_{n} (M_n)^d_{sa}$ be a matrix convex set.
Then
\[
 \cS_1 \subseteq \ol{\bB} \Longrightarrow   \cS \subseteq \sqrt{d} \, \fB \subseteq \sqrt{d} \, \fD
\]
and
\[
\ol{\bB} \subseteq \cS_1 \Longrightarrow \fD \subset \fB^\bullet \subseteq \sqrt{d} \, \cS .
\]
Moreover, the constant $\sqrt{d}$ is the optimal constant in both implications.
\end{theorem}

\begin{proof}
Suppose that $\cS_1 \subseteq \ol{\bB}$ and that $X \in \cS$.
Then $X \in \Wmax{}(\ol{\bB})$ (see Remark \ref{rem:monotone}), so $\sum_j a_j X_j \leq I$ for all $a \in \ol{\bB}$.
In particular $\pm X_j \leq I$, equivalently $X_j^2 \leq I$ for all $j$.
Thus $\sum_j X_j^2 \leq d I$, meaning that $\frac{1}{\sqrt d}X \in \fB$, as required.

To obtain the second implication we use polar duality.
If $\ol{\bB} \subseteq \cS_1$, then $(\cS^\bullet)_1 \subseteq \ol{\bB}$, so by the first implication
\[
\cS^\bullet  \subseteq \sqrt{d} \, \fB.
\]
Applying the polar dual again, we obtain
\[
(\sqrt{d} \fB)^\bullet \subseteq \cS^{\bullet \bullet} = \cS .
\]
For $\cS^{\bullet\bullet} = \cS$, note that $\ol{\bB} \subseteq \cS_1$ implies that $0 \in \cS$ so one can invoke Lemma \ref{lem:bipolar}.
The result therefore follows from
\[
(\sqrt{d} \fB)^\bullet = \frac{1}{\sqrt{d}} \fB^\bullet \supseteq \frac{1}{\sqrt{d}} \fD .
\]

Now by Lemma~\ref{lem:example}, there is a $d$-tuple of real Hermitian matrices $B$ in $\Wmax{}(\ol{\bB})$
such that $\| \sum B_i \otimes B_i \| = d$.
Since $B$ is real, we have $B = \ol{B}$.
Therefore it is clear that $c = 1/\sqrt d$ is the largest constant so that $cB \in \fD$.
The other inequality is also sharp by duality.
\end{proof}

\begin{remark}
We can summarize the previous theorem as
\be\label{eq:ball containment}
\Wmax{}(\ol{\bB}) \subseteq \sqrt{d} \, \fD \subseteq {d} \, \Wmin{}(\ol{\bB}).
\ee
From this we obtain Corollary \ref{cor:mainrelax2} for the case where $\cS_1 = \ol{\bB}$.
\end{remark}

By Theorem \ref{thm:MBP}, $\fB \subset \fD \subseteq \sqrt{d} \Wmin{}(\ol{\bB})$.
We therefore have the following corresponding dilation result.

\begin{corollary}
For $X=(X_1,...,X_d) \in (M_n)_{sa}^d$, if $\|\sum_{j=1}^d X_j \otimes \ol{X_j} \|\leq 1$ (in particular, if $\sum_{j=1}^d X_j^2 \leq I$), then there exists $T=(T_1,...,T_d)$ commuting self-adjoint matrices such that $\sigma(T) \subseteq \ol{\bB}$ and $\sqrt{d}T$ dilates $X$.
\end{corollary}

\section{Relaxations and inclusions}\label{sec:relaxations}

\subsection{Commutability index and inclusion scale}
For a matrix convex set $\cS$ containing $0$, we denote by $\cF_\cS$ the collection of tuples $T=(T_1,...,T_d)$ of commuting normal operators on some Hilbert space $H$, such that $\sigma(T) \subseteq \cS_1$.
Following \cite[Section 8.2]{HKMS15} (where the case of interest was a free spectrahedron $\cS = \cD_A$), we define
\[
\Gamma_\cS(n) = \{ \ t \geq 0 \ | \ \text{if } X \in \cS_n \text{ then } tX \text{ dilates into } \cF_\cS  \} ,
\]
and
\[
\Omega_\cS(n) = \{ t \geq 0 \ | \ \text{if } \cS_1 \subseteq \cT_1 \text{ then } t\cS_n \subseteq \cT_n  \} ,
\]
(where $\cT$ ranges over all matrix convex sets).
Note that we require that $0 \in \cS$ to make sure that these sets are not empty.
We define the {\em commutability indices of $\cS$} to be $\tau_\cS(n) =  \sup\Gamma_\cS(n)$,
and we define the {\em inclusion scales of $\cS$} to be $\rho_\cS(n) =  \sup\Omega_\cS(n)$ for each $n\ge1$.

Here we are interested in rank independent bounds, so we define the {\em rank independent commutability index} and {\em inclusion scales}, respectively, to be
\[ \tau_\cS = \inf_n \tau_\cS(n) \qand  \rho_\cS = \inf_n \rho_\cS(n) . \]
In \cite[Theorem 8.4]{HKMS15}, it was proved that $\tau_{\cD^{sa}_A}(n) = \rho_{\cD^{sa}_A}(n)$ for all $n$, where $A$ is a tuple of matrices (under the assumption that $\cD^{sa}_A(1)$ is bounded).
It follows then that $\tau_{\cD_A} = \rho_{\cD_A}$ in that case.
We provide here a quick proof of the latter in the general case, and also observe that these numbers are positive in the case that $\cS = \cD_A$ for $A$ a tuple of bounded operators.
Recall that a closed matrix convex set $\cS$ in $\cup_n M_n^d$ or $\cup_n (M_n)^d_{sa}$ has the form $\cS = \cD_A$ or $\cD_A^{sa}$, if and only $0 \in \operatorname{int}(\cS_1)$ (Proposition \ref{prop:LMIrange}).

\begin{theorem}\label{thm:ciis}
Let $\cS$ be a closed matrix convex set in $\cup_n (M_n)^d_{sa}$ containing $0$.
Then $\tau_\cS = \rho_\cS$.
If $\cS$ is bounded and $0 \in \operatorname{int}\cS_1$, then $\tau_\cS > 0$.
If $\cS_1$ satisfies the conditions of Theorem $\ref{thm:mainrelax}$, then $\tau_\cS \geq \frac{1}{C}$.
In particular, if $\cS_1 \subseteq \bR^d$ is invariant under projections onto an isometric tight frame, then $\tau_\cS \geq \frac{1}{d}$.
\end{theorem}

\begin{proof}
By Theorem~\ref{thm:scaled-inclusion-min}, $\tau_\cS = \rho_\cS$.
We will establish this in the self-adjoint setting. The nonself-adjoint setting is handled similarly.
If $\cS$ is bounded and $0 \in \operatorname{int}\cS_1$, then $\cS$ is contained in a large cube $R\fC$ and $\cS_1$ contains a small cube $[-r,r]^d$.
By Theorem \ref{thm:1dilation}, for every $X \in \cS$, the scaled tuple $\frac{r}{dR}X$ has a normal dilation with spectrum contained in $[-r,r]^d$.
Thus,
\[ \tau_\cS = \inf_n \tau_\cS(n)  \geq \frac{r}{dR} > 0 . \]

The final assertion follows from Theorem \ref{thm:mainrelax} together with the equality $\tau_\cS = \rho_\cS$ and our previous observations.
\end{proof}

\begin{remark}
Example \ref{ex:non-scalable} shows that the assumptions on boundedness and $0 \in \operatorname{int}(\cS)$ are necessary.
\end{remark}

Using duality, we obtain the following inclusion scale in the converse direction.

\begin{theorem}\label{thm:ciis_inverted}
Let $\cS$ be a closed matrix convex set in $\cup_n (M_n)^d_{sa}$.
If $\cS$ is bounded and $0 \in \operatorname{int}\cS_1$, then there exists a positive number $\sigma_\cS$ such that for every matrix convex set $\cT$,
\[
\cT_1 \subseteq \cS_1 \Rightarrow \sigma_\cS \cT \subseteq \cS.
\]
In fact, one can take $\sigma_\cS = \tau_{\cS^\bullet} = \rho_{\cS^\bullet}$.
If in addition $\cS_1$ satisfies the conditions of Theorem $\ref{thm:mainrelax}$, then $\sigma_\cS \geq \frac{1}{C}$.
In particular, if $\cS_1 \subseteq \bR^d$ is invariant under projections onto an isometric tight frame,
then $\sigma_\cS \geq \frac{1}{d}$.
\end{theorem}

\begin{proof}
Consider $\cS^\bullet$.
Since $\cS$ is bounded and has $0$ in the interior, $\cS^\bullet$ has these properties too.
Moreover, if $\cS_1$ is $\frac{1}{C}\lambda$-invariant, then $\cS_1^\bullet$ is $\frac{1}{C}\lambda^t$-invariant, and vice versa.

Now suppose that $\cT_1 \subseteq \cS_1$.
Then $\cS_1^\bullet \subseteq \cT_1^\bullet$, and by Theorem \ref{thm:ciis} we have $\tau_{\cS^\bullet} \cS^\bullet \subseteq \cT^\bullet$.
Applying the polar dual once more, and letting $\sigma_{\cS} = \tau_{\cS^\bullet}$, we have
\[
\sigma_\cS \cT \subseteq \sigma_\cS \cT^{\bullet \bullet} \subseteq \cS^{\bullet \bullet} = \cS. \qedhere
\]
\end{proof}

\subsection{Consequences for CP maps}\label{sec:CPMaps}

As observed in \cite[Section 1.4]{HKMS15}, the inclusion scales $\rho_{\cS}(n)$ can be interpreted as the maximal constants by which all unital positive maps into $M_n^{sa}$ can be ``scaled'' to become completely positive.
We present here results of a similar nature.

\begin{theorem}\label{thm:PtoUCP}
Let $A \in \cB(H)_{sa}^d$, let $\cS = \cW(A)$, and suppose that $0 \in \cS$.
If $B\in \cB(K)^d_{sa}$ and the map $S_A \to S_B$ given by
\[
I \mapsto I \quad , \quad A_i \mapsto B_i \,\, , \,\, i=1, \ldots, d ,
\]
is positive, then the map given by
\[
I \mapsto I \quad , \quad A_i \mapsto \sigma_{\cS} B_i \,\, , \,\, i=1, \ldots, d ,
\]
is completely positive. Moreover $\sigma_{\cS}$ is the optimal constant that works for all $B$.

Likewise, if $B\in \cB(K)^d_{sa}$ and the map $S_B \to S_A$ given by
\[
I \mapsto I \quad , \quad B_i \mapsto A_i \,\, , \,\, i=1, \ldots, d ,
\]
is positive, then the map given by
\[
I \mapsto I \quad , \quad B_i \mapsto \rho_{\cS} A_i \,\, , \,\, i=1, \ldots, d ,
\]
is completely positive.  Moreover $\rho_{\cS}$ is the optimal constant that works for all $B$.

If $0 \in \operatorname{int}\cW_1(A)$, then $\rho_{\cS}>0$ and $\sigma_{\cS} > 0$.
If $\cW_1(A)$ satisfies the conditions of Theorem $\ref{thm:mainrelax}$, then $\rho_{\cS}\geq \frac{1}{C}$ and $\sigma_{\cS} \geq \frac{1}{C}$.
In particular, if $\cW_1(A)$ is invariant under projections onto an isometric tight frame,
then both $\rho_\cS$ and $\sigma_\cS$ are bounded below by $\frac{1}{d}$.
\end{theorem}

\begin{remark}
Equivalently, this theorem could have been phrased in terms of $\cD_A$ instead of $\cW(A)$, where the condition $0 \in \operatorname{int}\cW_1(A)$ would be replaced by the equivalent condition that $\cD_A$ is bounded.
\end{remark}

\begin{proof}
Having a tuple $B \in \cB(K)^d_{sa}$ at hand, consider the matrix convex set $\cT = \cW(B)$.
If there is a unital positive map $S_A \to S_B$ determined by $A_i \mapsto B_i$, then $\cT_1 \subseteq \cS_1$ according to Theorem \ref{thm:ExistUCP_matRan}.
By Theorem \ref{thm:ciis_inverted}, $\sigma$ is strictly positive, has the stated lower bound under the symmetry assumption, and satisfies $\sigma_\cS \cT \subseteq \cS$.
But $\sigma_\cS \cT = \cW(\sigma_\cS B)$.
By Theorem \ref{thm:ExistUCP_matRan}, there is a UCP map $S_A \to S_B$ sending $A$ to $\sigma_\cS B$ as required.
The rest of the theorem is proved in a similar manner, using Theorem \ref{thm:ciis}.
\end{proof}

\subsection{The matrix cube and the matrix diamond}\label{subsec:cubediamond}

Let
\[
D_d = \{x \in \bR^d : \sum|x_j| \leq 1\} .
\]
This is the unit ball of $\ell^1_d$, and is the polar dual (in the usual, Banach space sense) of the unit ball of $\ell^\infty_d$, which is $[-1,1]^d$.
We may call $\Wmax{}(D_d)$ the {\em matrix diamond}.
By Corollary \ref{cor:bounded}, we have $(\Wmin{}([-1,1]^d))^\bullet = \Wmax{}(D_d)$ and $(\Wmax{}([-1,1]^d))^\bullet = \Wmin{}(D_d)$.

\begin{theorem}\label{thm:dilationD}
For every $d$-duple $A \in \cB(H)_{sa}^d$ for which $\sum_j \epsilon_j A_j \leq I$ for all $\epsilon \in \{-1,1\}^d$, there exists a dilation consisting of commuting self-adjoint contractions.
\end{theorem}
\begin{proof}
Construct the dilation $T$ as in the proof of Theorem \ref{thm:1dilation}.
Recall that $T_j$ has the form $\sum_i A_i \otimes S_{ij}$, where $S_{1j}, \ldots, S_{dj}$ is a family of commuting self-adjoint unitaries.
Thus $T_j$ is unitarily equivalent to a direct sum of operators of the form $\sum_i \epsilon_i A_i$, where $\epsilon \in \{-1,1\}^d$.
It follows from the assumption that $\|T\|\leq 1$.
\end{proof}

We obtain the following variant of Corollary \ref{cor:mainrelax}, which gives conditions for including the matrix diamond.

\begin{corollary}\label{cor:relaxationD}
For every $d \in \bN$,
\[
\Wmax{}(D_d) \subseteq \Wmin{}([-1,1]^d) .
\]
Hence for every matrix convex set $\cS$ we have
\[
[-1,1]^d \subseteq \cS_1 \quad\Longrightarrow\quad \Wmax{}(D_d) \subseteq \cS.
\]
\end{corollary}

\begin{remark}
Note that Corollary \ref{cor:mainrelax} gives
\[[-1,1]^d \subseteq \cS_1 \quad\Longrightarrow\quad \frac{1}{d}\Wmax{}([-1,1]^d) \subseteq \cS .
\]
However $\frac{1}{d}\Wmax{}([-1,1]^d)  \subsetneq \Wmax{}(D_d)$ because of a strict inclusion in the first level $\frac{1}{d} [-1,1]^d\subsetneq D_d$.
\end{remark}

\begin{proof}
For the first inclusion, note that $X \in \Wmax{}(D_d)$ if and only if $\sum_j \epsilon_j X_j \leq I$ for all $\epsilon \in \{-1,1\}^d$.
If $X \in \Wmax{}(D_d)$, then by the theorem above, $X$ has a commuting self-adjoint dilation $T$ consisting of contractions.
Thus $\sigma(T) \subseteq [-1,1]^d$.
We conclude that $X \in  \Wmin{}([-1,1]^d)$ (recall Definition \ref{def:min and max} and Proposition \ref{min and max}).
\end{proof}

\begin{remark}\label{rem:relaxationD}
It is worth noting that the diamond $D_d$ is a convex polytope and it is therefore determined by finitely many linear inequalities.
The set $\Wmax{}(D_d)$ therefore has a representation as a free spectrahedron $\Wmax{}(D_d) = \cD^{sa}_A$ for an appropriate $d$-tuple of self-adjoint matrices $A$.
Therefore, when $\cS = \cD^{sa}_B$ is also a free spectrahedron (determined by a tuple of matrices $B$), the problem of determining whether or not $\Wmax{}(D_d) \subseteq \cS = \cD_B^{sa}$ is amenable to the algorithms described in \cite{HKM13} (see also \cite{AG15a,AG15b}).
Thus we obtain a relaxation for the ``matrix cube problem'' of Ben-Tal and Nemirovski \cite{BTN02} which is of widespread interest: to rule out the containment of the cube $[-1,1]^d$ inside a spectrahedron $\cD_B^{sa}$ it suffices to rule out the containment of $\Wmax{}(D_d)$ inside $\cD^{sa}_B$.
This relaxation is more precise than the rank independent relaxation (checking whether $\fC^{(d)} \subseteq d\cD_B^{sa}$) given by Corollary \ref{cor:mainrelax}, and perhaps has some advantage over the rank dependent relaxation given by \cite[Theorem 1.6]{HKMS15} (checking whether $\fC^{(d)} \subseteq \vartheta(n) \cD_B^{sa}$).
\end{remark}

\begin{theorem}\label{thm:dilationD2}
For every $d$-duple $A \in \cB(H)_{sa}^d$ of contractions, there exists a dilation $T$ consisting of commuting self-adjoint operators such that $\sum_j \epsilon_j T_j \leq d I$ for all $\epsilon \in \{-1,1\}^d$.
\end{theorem}

\begin{proof}
Construct the dilation $T$ as constructed in the proof of Theorem \ref{general_dilation}, where we choose for the family $\{\lambda^{(1)}, \ldots, \lambda^{(d)}\}$ the operators $\lambda^{(k)} = d e_k e_k^*$ (this particular choice of $\lambda$s gives rise to the dilation discovered in \cite[Section 14]{HKMS15}).
Note that in that Theorem we constructed a dilation for a tuple of matrices $X$, and now we are working with operators, but the proof works the same.

Thus, $T_j$ has the form $0 \oplus \cdots 0 \oplus d A_j \oplus 0 \cdots \oplus 0$.
It follows that $\sum_j \epsilon_j T_j \leq d I$ for all $\epsilon \in \{-1,1\}^d$, as required.
\end{proof}

\begin{corollary}\label{cor:relaxationD2}
For every $d \in \bN$,
\[
\fC^{(d)} = \Wmax{}([-1,1]^d) \subseteq d \Wmin{}(D_d) .
\]
In particular, for every matrix convex set $\cS$ we have
\[
D_d \subseteq \cS_1 \Longrightarrow \fC^{(d)} \subseteq d\cS.
\]
\end{corollary}

\begin{remark}
Note that Corollary \ref{cor:mainrelax2} gives
\[
[-1,1]^d \subseteq \cS_1 \Rightarrow \fC^{(d)} \subseteq d \cS .
\]
Here we get the same conclusion from the weaker assumption $D_d \subseteq \cS_1$.
\end{remark}
\begin{proof}
Let $X \in \fC^{(d)}$.
By the above theorem, $X$ has a commuting self-adjoint dilation $T$ satisfying $\sum_j \epsilon_j T_j \leq d I$ for all $\epsilon \in \{-1,1\}^d$.
It follows that $\sigma(T) \subseteq d D_d$,
and we conclude that $X \in  d\Wmin{}(D_d)$.
\end{proof}

\subsection{Sharpness of the inclusion scale}
It is clear that the constant $d$ appearing in Corollary \ref{cor:relaxationD2} (and therefore also in Theorem \ref{thm:dilationD2}) is sharp, because $d$ is the smallest constant $c$ such that $[-1,1] \subseteq c D_d$.
The diameter of $[-1,1]^d$ is $2 \sqrt{d}$, and the diameter of $d D_d$ is $2d$.
The ratio between the diameter of $D_d$ and $[-1,1]^d$ is also $\sqrt{d}$.
Can one obtain the same results when $dD_d$ is replaced by some set of diameter less that $2d =\sqrt{d} \times \operatorname{diam} \fC$?
The following example gives a negative answer.

\begin{example}
Let $B$ be as in Lemma \ref{lem:example}.
Now $B \in \fC^{(d)}$, but we saw in Example \ref{ex:dsharp} that $B \notin \rho \Wmin{}(\ol{\bB})$ for any $\rho < d$, so $B \notin  \Wmin{}(C)$ for any symmetric convex set $C$ of diameter less than $2d =\sqrt{d} \times \operatorname{diam} \fC$.
\end{example}

\subsection{Generalization to balanced convex polytopes generated by tight frames}\label{subsec:gen}
Here we generalize the results of this section to the case where the diamond and the cube are replaced by a convex polytope $K$ generated by a tight frame.
As a consequence, we obtain that the conditions required by Theorem \ref{relaxation and symmetry} hold for such convex polytopes, but for such polytopes we get a stronger conclusion.

For a convex polytope $K$ denote by $K'$ its scalar polar dual in $\bR^d$:
\[
K' = \{x \in \bR^d : {\textstyle\sum_j} x_j y_j  \leq 1 \FORAL y \in K  \}.
\]
Recall that a set $\{v^{(1)}, \ldots, v^{(N)}\} \subseteq \bR^d$ is called a {\em tight frame} if there is a constant $\sigma$ (called the constant of the frame) for all $x \in \bR^d$,
\[
\sum_i |\langle v^{(i)}, x \rangle|^2 = \sigma \|x\|^2.
\]
This is equivalent to the condition that $\sum_i v^{(i)} v^{(i)*} = \sigma I$.
By \cite[Lemma 2.1]{RW02}, we must have that $\sigma = \frac{\sum_i \|v^i \|^2}{d}$.

\begin{theorem}
Let $\{v^{(1)}, \ldots, v^{(N)}\} \subseteq \bR^d$ be a tight frame with constant $\sigma = \frac1d  \sum_i \|v^{(i)}\|^2$.
For $c_1,...,c_N > 0$,  set $K = \conv \{\pm c_i v^{(i)} : i=1, \ldots, N\}$.
Define
\[
\kappa = \frac{\sigma \min_i c_i^3}{\sum_i c_i} = \frac{ \sum_i \|v^{(i)}\| }{\sum_i c_i}\cdot \frac{\min_i c_i^3}{d}.
\]
Then
\[
\kappa \Wmax{}(K') \subseteq \Wmin{}(K).
\]
\end{theorem}

\begin{remark}
If $\{v^{(i)}, \ldots, v^{(N)}\}$ is taken to be an orthonormal basis for $\bR^d$ (so $N=d$) and $c_1 = \cdots = c_d = 1$,
then $K$ is the diamond and $K'$ is the cube.
Moreover $\sigma = 1$ and $\kappa = \frac{1}{d}$. For $c_i=1$, and we recover Corollary \ref{cor:relaxationD2}.

On the other hand, if $\{v^{(1)}, \ldots, v^{(N)}\}$ is taken to be the corners of the cube (so $N=2^d$),
then for $c_i=1$, $K$ is the cube, $K'$ is the diamond, $\sigma = 2^d$ and $\kappa = 1$.
We recover Corollary \ref{cor:relaxationD}.
\end{remark}

\begin{proof}
Let $X \in \Wmax{}(K')$.
By definition of $\Wmin{}(K)$, we need to construct a commuting self-adjoint dilation $T$ for $X$ such that $\sigma(\kappa T) \subseteq K$.
We wish to use the construction from Theorem \ref{general_dilation}.

For $m=1, \ldots, N$, define the rank one matrices
\[
\lambda^{(m)} = b_m v^{(m)} v^{(m)*},
\]
where $b_m= \sigma^{-1} \frac{\sum_i c_i}{c_m}$.
Now define diagonal matrices
\[
 S_{ij} = \diag(\lambda^{(1)}_{ij}, \ldots, \lambda^{(k)}_{ij}) \qfor i,j=1, \ldots, d,
\]
and put
\[
T_i = \sum_j X_j \otimes S_{ij}.
\]
As in the proof of Theorem \ref{general_dilation}, $T$ is a commuting self-adjoint tuple.
Recall from that proof that if we show that $I \in \conv\{\lambda^{(1)}, \ldots, \lambda^{(k)}\}$, it would follow that $T$ is a dilation for $X$.
Now if we put
\[
a_m = (\sigma b_m)^{-1} = \frac{c_m}{\sum_i c_i} ,
\]
then $\sum a_m = 1$ and
\[
\sum_m a_m \lambda^{(m)} = \sum_m \sigma^{-1} v^{(m)} v^{(m)*} = I.
\]

To show that $\sigma(\kappa T) \subseteq K$, it suffices to show that $\kappa T$ satisfies the inequalities determining $K$.
Let $u \in K'$.
Then
\begin{align*}
\sum_i u_i \kappa T_i
&=  \kappa\sum_i u_i \sum_j X_j \otimes S_{ij} \\
&= \kappa \sum_i u_i \sum_j \oplus_{m=1}^k  \lambda_{ij}^{(m)} X_j  \\
&= \kappa \oplus_{m=1}^k \sum_j  \sum_i u_i b_m v_i^{(m)} v_j^{(m)} X_j \\
&= \oplus_{m=1}^k \langle \kappa b_m v^{(m)}, u \rangle \cdot \sum_j v_j^{(m)} X_j .
\end{align*}
Thus we need to show that for all $m$,
\[
|\langle \kappa b_m v^{(m)}, u \rangle|\|\sum_j  v_j^{(m)} X_j\| \leq 1 .
\]
Since $X \in \Wmax{}(K')$, we have that
\[
 \|\sum_j v_j^{(m)} X_j\| = \frac{1}{c_m} \|\sum_j c_m v_j^{(m)} X_j\| \leq \frac{1}{c_m} .
\]
But we also have that
\[
|\langle \kappa b_m v^{(m)}, u \rangle| = \kappa b_m \cdot |\langle v^{(m)}, u \rangle| = \frac{\kappa b_m}{c_m} \cdot |\langle c_m v^{(m)}, u \rangle| \leq \frac{\kappa b_m}{c_m} \leq c_m .
\]
So we are done.
\end{proof}

\bibliographystyle{amsplain}

\end{document}